\documentclass{article}
\usepackage[a4paper,outer=3.5cm,inner=3.5cm,total={14cm,21.7cm}]{geometry} 
\usepackage[english]{babel}
\usepackage[utf8]{inputenc}
\usepackage{amsmath}
\usepackage{amssymb}
\usepackage{amsthm}
\usepackage{enumerate}
\usepackage[shortlabels]{enumitem}
\usepackage[colorlinks=true]{hyperref}
\usepackage{mathtools}
\usepackage{tikz-cd}
\usepackage{stmaryrd}

\newcommand{\overbar}[1]{\mkern 1.5mu\overline{\mkern-1.5mu#1\mkern-1.5mu}\mkern 1.5mu}

\newtheorem{thm}{Theorem}[section]
\newtheorem{Theorem}[thm]{Theorem}
\newtheorem{Lemma}[thm]{Lemma}
\newtheorem{Proposition}[thm]{Proposition}
\newtheorem{Corollary}[thm]{Corollary}
\newtheorem{Conjecture}[thm]{Conjecture}

\theoremstyle{definition}
\newtheorem{Definition}[thm]{Definition}
\newtheorem{Question}[thm]{Question}
\newtheorem{Example}[thm]{Example}
\newtheorem{Remark}[thm]{Remark}
\newtheorem{notation}[thm]{Notation}

\newcommand{\C}{\mathbb{C}}
\newcommand{\F}{\mathbb{F}}
\newcommand{\G}{\mathbb{G}}
\newcommand{\N}{\mathbb{N}}
\renewcommand{\O}{\mathcal{O}}
\renewcommand{\P}{\mathbb{P}}
\newcommand{\Q}{\mathbb{Q}}
\newcommand{\Z}{\mathbb{Z}}
\newcommand{\m}{\mathfrak{m}}

\newcommand{\uZ}{\underline{\Z}}
\newcommand{\uZp}{\underline{\Z}_p}
\newcommand{\uQp}{\underline{\Q}_p}

\newcommand{\whZ}{{\wh\Z}}
\newcommand{\RGamma}{\mathrm R\Gamma}
\newcommand{\RHom}{\mathrm R\Hom}

\newcommand{\Hom}{\mathrm{Hom}}
\newcommand{\Map}{\operatorname{Map}}
\newcommand{\Mor}{\operatorname{Mor}}
\newcommand{\Mapc}{\operatorname{Map}_{\cts}}
\newcommand{\Ext}{\operatorname{Ext}}
\newcommand{\HOM}{\underline{\mathrm{Hom}}}
\newcommand{\End}{\operatorname{End}}
\newcommand{\Aut}{\operatorname{Aut}}

\newcommand{\Perf}{\operatorname{Perf}}
\newcommand{\Spa}{\operatorname{Spa}}

\newcommand{\Spd}{\operatorname{Spd}}

\newcommand{\cts}{{\operatorname{cts}}}
\newcommand{\an}{\mathrm{an}}
\newcommand{\et}{{\mathrm{\acute{e}t}}}

\newcommand{\qproet}{\mathrm{qpro\acute{e}t}}
\newcommand{\fet}{{\mathrm{f\acute{e}t}}}

\newcommand{\aeq}{\stackrel{a}{=}}

\newcommand{\wh}{\widehat}
\newcommand{\bOx}{\overbar{\O}^{\times}}
\newcommand{\Pic}{\operatorname{Pic}}
\newcommand{\uP}{\underline{\Pic}}
\newcommand{\NS}{\mathrm{NS}}
\newcommand{\wt}{\widetilde}
\newcommand{\perf}{\mathrm{perf}}
\renewcommand{\diamond}{\diamondsuit}

\newcommand{\rk}{\mathrm{rk}}
\newcommand{\Char}{\operatorname{char}}
\renewcommand{\lim}{\varprojlim}
\newcommand{\Rlim}{\mathrm{R}\!\lim}
\newcommand{\cH}{{\ifmmode \check{H}\else{\v{C}ech}\fi}}
\newcommand{\tf}{[\tfrac{1}{p}]}
\renewcommand{\tt}{\mathrm{tt}}

\newcommand*\isomarrow{%
	\xrightarrow{\raisebox{-0.35em}{\smash{\ensuremath{\sim}}}}
}  

\tikzset{
	labelrotate/.style={anchor=south, rotate=90, inner sep=.5mm}} 

\makeatletter
\newcommand{\doublewidetilde}[1]{{%
		\mathpalette\double@widetilde{#1}%
}}
\newcommand{\double@widetilde}[2]{%
	\sbox\z@{$\m@th#1\widetilde{#2}$}%
	\ht\z@=.85\ht\z@
	\widetilde{\box\z@}%
}
\def\mathcenterto#1#2{\mathclap{\phantom{#1}\mathclap{#2}}\phantom{#1}}
\let\old@widetilde\doublewidetilde
\def\widetildeto#1#2{\mathcenterto{#2}{\doublewidetilde{\mathcenterto{#1}{#2\,}}}}
\makeatother
\newcommand{\wtt}[1]{\widetildeto{I}{#1}}

\begin{document}
\title{pro-\'etale uniformisation of abelian varieties}
\author{Ben Heuer}
\date{}
\maketitle
\begin{abstract}
	For an abelian variety $A$ over an algebraically closed non-archimedean field $K$ of residue characteristic $p$, we show that the isomorphism class of the pro-\'etale perfectoid  cover $\wt A=\varprojlim_{[p]}A$ is locally constant as $A$ varies $p$-adically in the moduli space. This gives rise to a pro-\'etale uniformisation of abelian varieties as diamonds
	\[A^\diamond=\wt A/T_pA\]
	 that works uniformly for all $A$ without any assumptions on the reduction of $A$.
	 
	 More generally, we determine all morphisms between pro-finite-\'etale covers of abeloid varieties.  For example, over $\C_p$, all abeloids can be uniformised in terms of universal covers that only depend on the isogeny class of the semi-stable reduction over $\overline{\F}_p$.
\end{abstract}

\setcounter{tocdepth}{2}
\section{Introduction}
Let $A$ be a complex abelian variety, then $A$ admits a uniformisation as a complex torus
\begin{equation}\label{eq:intro-complex-uniform}
A(\C)\cong \C^g/\Lambda
\end{equation}
in terms of its topological universal cover $\C^g$ and some lattice $\Lambda\cong \Z^{2g}$, where $g$ is the dimension of $A$. More generally, this works for connected compact complex Lie groups.

An analogous uniformisation exists in the rigid analytic category for an abelian variety $A$ over an algebraically closed non-archimedean field $K$ of residue characteristic $p$: By work of Raynaud, Bosch and L\"utkebohmert, there is a semi-abelian rigid space $E$, i.e.\ an extension
\[ 0\to T\to E\to B\to 0\]
of an abelian variety $B$ of good reduction by a rigid torus $T$ of rank $r$, and a discrete lattice $\Z^r\cong M\subseteq E(K)$ of rank $r$, such that $A$ is isomorphic as a rigid group variety to the quotient
\begin{equation}\label{eq:intro-rigid-uniform}
A\cong E/M.
\end{equation}
 The space $E$ is called a Raynaud extension. It can be thought of as a rigid analytic ``universal cover'' of $A$, in the sense that $E$ is connected and $H^1_{\an}(E,\uZ)=0$ (see \cite[Theorem~1.2.(c)]{BL-Degenerating-AV}).
This ``Raynaud uniformisation'' works more generally for abeloid varieties over $K$, i.e.\ connected smooth proper rigid group varieties over $K$, the rigid analogues of complex tori.

However, this is a ``uniformisation'' of $A$ in a weaker sense of the word:
While in the complex case, the universal cover $\C^g$ is isomorphic among all abelian varieties of dimension $g$, there is in contrast a very large number of different Raynaud extensions $E$ for abelian varieties of the same dimension.
 For example, if $A$ has good reduction, we simply have $A=E=B$, so one could say that there is not really any uniformisation in this case.  In particular, it is in general not possible to study $p$-adic families of abelian varieties in terms of the same rigid cover, as one might want to do to study, say, variations in the Picard rank.
 
The goal of this article is to prove that in the category of diamonds over $K$ recently introduced by Scholze \cite{etale-cohomology-of-diamonds}, there is a new kind of ``pro-\'etale uniformisation'' of abeloids which at least in this sense is stronger than rigid analytic uniformisation:

Let $A$ be an abeloid variety over $K$, for example an abelian variety. Consider $A$ as a diamond over $K$.
 By \cite[Theorem~1]{perfectoid-covers-Arizona}, there is then a perfectoid space $\wt A$ such that
\[\widetilde{A}= \textstyle\varprojlim_{[p]}A\]
as $v$-sheaves. The pro-\'etale cover $\wt A\to A$ has a topological universal property which we take as an analogue of saying that $\wt A$ is simply connected in the $v$-topology: It is connected, and for the pro-constant $v$-sheaf $\uZp:=\varprojlim {\Z/p^n\Z}$ we have 	(see  \cite[Proposition~4.2]{heuer-Picard-good-reduction})
	\[H^i_v(\wt A,\uZp)=0\quad \text{for }i>0.\]
	We thus get a uniformisation of $A$ as a diamond (see \cite[\S5.1]{perfectoid-covers-Arizona}): an isomorphism of $v$-sheaves
	\begin{equation}\label{eq:A=wt A/T_pA}
A^\diamondsuit=\wt A/T_pA
\end{equation}
in topological analogy to the complex uniformisation \eqref{eq:intro-complex-uniform} and the rigid uniformisation \eqref{eq:intro-rigid-uniform}. 

There is variant of this diamantine uniformisation using instead the adelic cover
\[\wtt A:=\textstyle\varprojlim_{[N],N\in\N}A,\]
which is also perfectoid. This is the universal pro-finite-\'etale cover in the sense of \cite[\S4.3]{heuer-v_lb_rigid}. It enjoys the analogous topological universal property for $\underline{\wh \Z} =\varprojlim_{N\in\N} \Z/N\Z$-coefficients,
and we get a diamantine uniformisation $A^\diamond=\wtt A/TA$ using the adelic Tate module $TA$.

\subsection{Isomorphisms between universal covers}
Our main result is that this is in fact a uniformisation in a rather strong sense: In the pro-\'etale situation, there are many more isomorphisms between universal covers than in the rigid setup.
To make this precise, we first state the simplest case of our main result, for $\C_p$:
\begin{Theorem}\label{t:intro-adelic-covers-over-C_p}
	Let $A$ and $A'$ be two abeloids over $K=\C_p$ or $\C_p^\flat$. Let $T\to E\to B$ and $T'\to E'\to B'$ be the associated Raynaud extensions, and $\overbar{B},\overbar{B}'$ the special fibres of $B, B'$ over $\overline{\F}_p$.
	Then the following are equivalent:
	\begin{enumerate}
		\item There is an isomorphism of perfectoid groups $\wtt A\isomarrow \wtt A'$.
		\item There is an isogeny $\overbar B\to \overbar B'$ and $\rk(T)=\rk(T')$.
	\end{enumerate}
\end{Theorem}
In other words, while the Raynaud uniformisation of $A$ is in terms of the semi-abelian rigid space $E$, the diamantine uniformisation of $A$ over $\C_p$ only depends on the semi-stable reduction of $A$ up to isogeny. In particular, the isomorphism class of $\wtt A$ is locally constant in $A$, and like in the complex case, one can uniformise families of abeloid varieties by quotienting out different $p$-adic lattices out of the same covering space $\wtt A$. As a corollary, there are many more ``pro-\'etale isogenies'' between abeloid varieties, i.e.\ diagrams with pro-\'etale covers
\[ \begin{tikzcd}[row sep ={0cm}]
	& Z \arrow[ld] \arrow[rd] &  \\
	A &  & A'
\end{tikzcd}\]
by some adic space $Z$, than there are \'etale isogenies of this form. In fact, working in diamonds, one can even take $A$ and $A'$ to be of different characteristics (see Theorem~\ref{t:tilting-abeloids}).

To give the precise statement over general fields, we first need to introduce some notation:
Let $\O_K$ be the ring of integers of $K$, let $\m$ be the maximal ideal, and $k$ the residue field. The Raynaud extension $E$ corresponds to a formal semi-abelian scheme over $\O_K$ whose special fibre over $k$ is a semi-abelian scheme $\overbar{E}$. This is the semi-stable reduction of $A$.

The following theorems now first identify the isomorphism classes of the perfectoid cover
\[ \wt E=\textstyle\varprojlim_{[p]}E\]
in terms of the reduction and then use this to describe the isomorphism classes of $\wt A$.
\begin{Theorem}\label{t:intro-Hom(E)}
	Let $E$ and $E'$ be Raynaud extensions. Then there is a natural isomorphism
	\[ \Hom(\wt E,\wt E')=\Hom(\overbar{E},\overbar{E}')\tf.\]
\end{Theorem}

For a rigid or perfectoid group $G$ we have introduced in \cite[\S2]{heuer-diamantine-Picard} the ``topological $p$-torsion subgroup'' $\wh G\subseteq G$, a sub-$v$-sheaf (often an adic subgroup) whose $K$-points are given by
	\[ \wh G(K)=\{x\in G(K)\mid p^n\Z\cdot x\xrightarrow{n\to \infty} 0\}=\Hom_{\cts}(\Z_p,G(K)).\]
\begin{Definition}
	Let $F,G$ be rigid or perfectoid groups. We say that two maps $f,g:F\to G$ agree up to topological $p$-torsion if $f-g$ factors through $\wh G$.
	 In particular, this gives a notion of what it means for a diagram of abelian $v$-sheaves to commute up to topological $p$-torsion.
\end{Definition}
\begin{Theorem}\label{t:intro-Hom(A,A')}
	Let $A$ and $A'$ be two abeloids with Raynaud uniformisations $A=E/M$ and $A'=E'/M'$.
	Choose any lifts $g:M\to \wt E$, $g':M'\to \wt E'$.
	Then we have:
	\[
		\Hom(\wt A,\wt A')={\left\{\begin{array}{@{}l@{}l}	
				\varphi_E&\in \Hom(\wt E,\wt E')\end{array}\middle|\, 
			\begin{array}{lr}
				\exists \,\varphi_M:M[\tfrac{1}{p}]\to M'[\tfrac{1}{p}]\\	
				\text{such that the following }\\
				\text{diagram commutes up}\\\text{to topological $p$-torsion:}\end{array}
			\quad
			{\begin{tikzcd}
					M[\tfrac{1}{p}] \arrow[d,"\varphi_M"'{yshift=0.36ex}] \arrow[r,"g"] & \wt E \arrow[d,"\varphi_E"'{yshift=-0.36ex}] \\
					M'[\tfrac{1}{p}] \arrow[r,"g'"]& \wt E'  
			\end{tikzcd}}
			\right \}}\]
	In particular, the following are equivalent:
	\begin{enumerate}
		\item There is an isomorphism of perfectoid groups $\wt A\isomarrow \wt A'$.
		\item \label{enum:t-cond-2} There is a $p$-isogeny $\overbar\varphi:\overbar  E\to \overbar E'$ between the semi-stable reductions and an isomorphism  $\varphi_M:M[\tfrac{1}{p}]\isomarrow M'[\tfrac{1}{p}]$ such that the above diagram commutes up to topological $p$-torsion,
		where the map on the right is the one corresponding to $\overbar\varphi$ via Theorem~\ref{t:intro-Hom(E)}.
		\end{enumerate}
\end{Theorem}
The fact that condition~2 only requires commutativity up to topological $p$-torsion means that the lattices $M$ and $M'$ only need to be ``$p$-adically close'' with respect to the ambient isomorphism $\wt E\to \wt E'$. In particular, $\wt A$ is ``locally constant'', e.g.\ in the following sense:
\begin{Corollary}\label{c:isom-class-in-moduli-space}
 The isomorphism class of $\wt A$ is locally constant on the Siegel moduli space $\mathcal A_g(K)$ endowed with the non-archimedean topology inherited from $K$.
 \end{Corollary}

Let us illustrate what this means explicitly at the example of Tate curves:

\begin{Corollary}
	For any $q,q'\in K^\times$ with $|q|,|q'|<1$, there is an isomorphism 
	\[\widetilde{\G_m/q^\Z}\cong \widetilde{\G_m/q'^\Z}\]
	of perfectoid groups if and only if $q'\in (1+\m)\cdot q^{\Z[\frac{1}{p}]}$. 
\end{Corollary}
\begin{Remark}
	If $A=B$ has good reduction, the direction $2)\Rightarrow 1)$ of Theorem~\ref{t:intro-Hom(A,A')} follows easily from the tilting equivalence, and it is $1)\Rightarrow 2)$ that requires more work.

However, already for Tate curves, the direction $2)\Rightarrow 1)$ has an entirely different reason, and seems to be a genuinely ``new'' phenomenon.
For example, in general the additional isomorphisms  cannot  be seen on any intermediate perfectoid cover of $A$ (see Example~\ref{ex:non-trivial-morph-between-A_1-and-A_2}).
\end{Remark}

\subsection{The conjectural pro-\'etale uniformisation of curves}

One reason why Theorem~\ref{t:intro-Hom(A,A')} is interesting is that it provides evidence that an even stronger non-archimedean uniformisation result might hold for smooth projective curves over $K$:

Recall that in complex geometry, connected smooth proper complex curves are equivalent to compact Riemann surfaces, for which there is the following uniformisation theorem: Any compact Riemann surface $X$ is isomorphic to a quotient $\wt X/\Gamma$ of its topological universal cover $\wt X$ by the free and properly discontinuous action of a  group $\Gamma$. Moreover, there are only three different universal covers up to holomorphic isomorphism: The universal cover is $\P^1$ if the genus of $X$ is $g=0$, it is $\C$ if $g=1$, or it is the open unit disc $\mathbb D$ if $g\geq 2$.

It is a long-standing question whether a similar description is possible in the non-archimedean world over $K$.
There are well-known partial analogues in the rigid setting: In the case of $g=1$ of elliptic curves, this is the rigid uniformisation of the Tate curve $\G_m/q^\Z$, which was Tate's original motivation to lay the foundations of rigid analytic geometry. More generally, such a rigid uniformisation can be constructed for connected smooth proper curves $C$ with totally degenerated reduction, due to Mumford \cite{MumfordCurve}\cite[\S2]{Lutkebohmert_RigidCurves}: Any such $C$ can be written as a quotient \[C=\Omega_{\Gamma}/\Gamma\] where $\Omega_{\Gamma}\subseteq \P^1$ is the complement of some open discs and $\Gamma
\subseteq \mathrm{PGL}_2(K)$ is a Schottky group. This is a beautiful result, but it is weaker than the complex description, as it only captures the totally degenerate case, and as $\Omega_{\Gamma}$ has many different isomorphism classes.

Motivated by the uniformisation of abeloids, we conjecture that in the diamantine setting, there is a stronger uniformisation result which is very close to the complex situation:
Let $C$ be a connected smooth projective curve over $K$ of genus $g\geq 1$. In analogy with \eqref{eq:A=wt A/T_pA}, any choice of base-point $x\in C(K)$ leads to a canonical diamantine uniformisation of $C$ \cite[\S5.3]{perfectoid-covers-Arizona}
\[ C^\diamondsuit = \wtt C/\pi_1(C,x)\]
where $\wtt C\to C$ is a certain universal pro-\'etale perfectoid cover, and $\pi_1(C,x)$ is the \'etale fundamental group. The following conjecture was independently also suggested by Hansen:
\begin{Conjecture}\label{c:isom-classes-of-covers-of-curves}
	The isomorphism class of $\wtt C$ is locally constant in the moduli space $\mathcal M_g(K)$ of connected smooth projective curves of genus $g$ over $K$ with its non-archimedean topology.
\end{Conjecture}
Corollary~\ref{c:isom-class-in-moduli-space} confirms this in the case $g=1$ of elliptic curves. For $g\geq 2$, the theorem shows that the result holds for the Jacobians, which  we see as evidence for the general case.

As we discuss in  \S\ref{s:curves}, the
Conjecture~\ref{c:isom-classes-of-covers-of-curves} has a very close relation to a $p$-adic analogue of the Ehrenpreis conjecture in complex geometry. This was first suggested by Daniel Litt \cite{Litt-MO-question} motivated by very similar considerations regarding uniformisation of  curves over $\Z_p^{\textrm{ur}}$:

In the case of genus $g=1$, the isomorphism classes of $\wtt C$ are given by the isogeny classes of the special fibre. 
Now for $g\geq 2$, there are ``more'' isogenies between curves than for $g=1$: Litt asks whether in fact any two smooth projective curves of genus $g\geq 2$ over $\overline{\F}_p$ might have a common finite \'etale cover. If that is the case, and 
Conjecture~\ref{c:isom-classes-of-covers-of-curves} is true, then  it seems conceivable that the cover $\wtt C$ is in fact isomorphic among \textit{all} curves of genus $g\geq 2$ over $\C_p$. This would give a uniformisation almost as strong as in the complex case, namely there would be the following set of isomorphism classes of universal covers of curves:
\begin{itemize}
	\item One is $\P^1$ for genus $g=0$.
	\item For genus $g=1$, there is one for all elliptic curves of semi-stable reduction, and one associated to each isogeny class of elliptic curves over $\overline{\F}_p$.
	\item There would be one capturing all curves of genus $g\geq 2$.
\end{itemize}
This would be very close to the complex uniformisation theorem, but also incorporate the phenomenon of good and bad reduction of elliptic curves, which has no analogue over $\C$.

\subsection{Morphisms of perfectoid spaces}
The perspective of uniformisation of curves motivates the question whether the answer to Theorem~\ref{t:intro-Hom(A,A')} changes if we forget the group structure on $\wt A$ for an elliptic curve $A$.

Recall that in the complex case, any injective entire function $f:\C\to \C$
is already linear, i.e.\ $f(x)=ax+b$ for some $a,b\in \C$.
Similarly, one can show in the rigid analytic setting that every morphism $f:E\to E'$
between Raynaud extensions is a group homomorphism up to translation.
We show that the analogue holds for the pro-\'etale universal covers:
\begin{Theorem}\label{t:morph-is-comp-intro}
	Let $A$ and $A'$ be two abeloids over $K$. Then any morphism of perfectoid spaces $\wtt A\to \wtt A'$ over $K$
	is the composition of a homomorphism and a translation.
\end{Theorem}
\subsection{Overview of proof of the Main Theorem}
To prove Theorem~\ref{t:intro-Hom(E)}, we first show that in the case of good reduction, we have
\[\Hom(\wt B_1,\wt B_2)=\Hom(\overbar B_1,\overbar B_2)\tf,\]
as follows: The tilting equivalence sets up a natural map from right to left. In order to see that this is an isomorphism, the main difficulty is showing that any morphism $\wt B_1\to\wt B_2$ has a formal model. We use the diamantine Picard functor of \cite{heuer-diamantine-Picard} to translate this into a question about line bundles on $\wt B_1\times B_2^\vee$, which is accessible by the methods of \cite{heuer-diamantine-Picard,heuer-Picard-good-reduction}.

In order to promote this to the statement of Theorem~\ref{t:intro-Hom(E)},
we use that there is an analytic short exact sequence
\[ 0\to \wt T\to \wt E\to \wt B \to 0.\]
The key observation is now that by \cite[{Theorem~4.1}]{heuer-Picard-good-reduction}, we have
$\Pic(\wt B)=\Pic(\overbar{B})\tf$, which means that for fixed $T$ and $B$, many extensions $E$ become isomorphic on universal covers. More precisely, let $M^\vee:=\Hom(T,\G_m)$, then $E$ as an extension of $B$ by $T$ defines an element
\[[E]\in \Ext_{\an}^1(B,T)= \Hom(M^\vee,B^\vee(K)),\]
whereas the isomorphism class of the universal cover is determined by the reduction
\[ [\wt E]\in\Ext^1_{\an}(\wt B,\wt T)=\Hom(M^\vee, \overbar{B}^\vee(k)\tf).\]
In order to deduce Theorem~\ref{t:intro-Hom(A,A')}, we study the limit $\wt E\to \wt A$ of $E\to A$:  Our starting point for this is \cite[Theorem~4.6.4]{perfectoid-covers-Arizona}, which implies that there is a short exact sequence on $\Perf_{K,v}$
\begin{equation}\label{eq:intro-E-A-Mp/M}
	\begin{tikzcd}
		0 \arrow[r] & \wt E \arrow[r] & \widetilde{A} \arrow[r] & M\otimes_{\Z} \uZp/\uZ\arrow[r] & 0.
	\end{tikzcd}
\end{equation}
The projection $E\to A$ thus becomes an injection $\wt E\to \wt A$. For this, we prove a  lifting result:

In the rigid analytic (or complex) setting, any group homomorphism $\varphi:A\to A'$ of abeloid varieties 
lifts uniquely to a homomorphism of the universal covers $E\to E'$ that determines $\varphi$ uniquely. We prove that a similar result holds in the perfectoid case:

\begin{Theorem}\label{t:intro-unique-lifting}
	Let $A$ and $A'$ be two abeloids over $K$ with associated Raynaud extensions $E$ and $E'$. Then
	any homomorphism of perfectoid groups $\varphi:\widetilde{A}\to \widetilde{A}'$
	restricts to a homomorphism $\wt E\to \wt E'$.
	Moreover, this restriction  determines $\varphi$ uniquely.
\end{Theorem}
A more conceptual way to think about this is in terms of the $M$-torsor
\[X:=M_p\times\wt E\to \wt A\]
where $M_p:=M\otimes \underline{\Z}_p$. Theorem~\ref{t:intro-unique-lifting} implies that for $n\gg0$, the map $[p^n]\varphi$ lifts uniquely to a morphism $X\to X'$.
 From this perspective, the reason why there are many isomorphisms of universal covers is that $\Aut(X)$ is large as there are many homomorphisms $M_p\to \wt E$. This is ultimately where the tolerance of topological $p$-torsion in Theorem~\ref{t:intro-Hom(A,A')} comes from.

The article ends with an appendix on a few tangential topics like $\Ext$-groups of $v$-sheaves.

\subsection*{Acknowledgements}
We would like to thank Peter Scholze for many helpful discussions and comments, in particular for his ideas for the proofs of Proposition~\ref{p:higher-cohomolgy-of-profinite=0} and equation \eqref{eq:Hom-UC-of-p-div-groups},  and for noticing the connection to the $p$-adic Ehrenpreis conjecture.
We would also like to thank Johannes Ansch\"utz, Alex Ivanov, Daniel Litt, Alex Torzewski and Peter Wear for useful conversations.

This work was funded by the Deutsche Forschungsgemeinschaft (DFG, German Research Foundation) under Germany's Excellence Strategy-- EXC-2047/1 -- 390685813. The author was supported by the DFG via the Leibniz-Preis of Peter Scholze.
\subsection*{Notation}
Throughout, we work over a complete algebraically closed non-archimedean field $K$ of residue characteristic $p>0$, with  ring of integers $\O_K$, maximal ideal $\m$ and residue field $k$. We fix a pseudo-uniformiser $\varpi\in\O_K$. If $\Char K=0$, we assume that $p\in \varpi\O_K$.
We use almost mathematics in the sense of Gabber--Ramero with respect to $(\O_K,\mathfrak m)$. For any $\O_K$-modules $M$, $M'$ we write $M\aeq M'$
if there is a natural almost isomorphism between them.

By a rigid space over $K$, we always mean an adic space of topologically finite type over $\Spa(K,\O_K)$.
By a rigid group over $K$, we mean a group object in rigid spaces over $K$, which by convention we always assume to be commutative. Similarly, a perfectoid group over $K$ will mean a commutative group object in the category of perfectoid spaces over $K$.
\section{Preparations: adic groups as $v$-sheaves}

The aim of this section is to collect some preliminaries that we will use freely throughout. These include generalities on the $v$-site from \cite[\S8]{etale-cohomology-of-diamonds} and the topological $p$-torsion subsheaf of \cite[\S2]{heuer-diamantine-Picard}. We also study the quotient sheaf $\uZp/\uZ$  in the short exact sequence \eqref{eq:intro-E-A-Mp/M}.

\subsection{Recollections on the $v$-site}
\begin{Definition}
	We denote by $\Perf_{K,v}$ the site of perfectoid spaces over $K$, equipped with the $v$-topology in the sense of \cite[Definition~8.1]{etale-cohomology-of-diamonds}. The tilting equivalence \cite[Theorem 5.2]{perfectoid-spaces} defines an equivalence of sites
	\[ \Perf_{K,v}=\Perf_{K^{\flat},v}\]
	which allows us to switch freely between characteristic $0$ and $p$.
	In particular, in either characteristic, we can identify the underlying category $\Perf_{K}$ with the slice category $\Perf|_{\Spd(K)}$ of perfectoid spaces in characteristic $p$ equipped with a morphism to $\Spa(K^\flat)$.
\end{Definition}
We recall a few result from \cite{etale-cohomology-of-diamonds}:
Given any analytic adic space $Z$ over $\Spa(K,\O_K)$, we get an associated presheaf $Z^\diamondsuit$ on $\Perf_{K,v}$ by sending any perfectoid space $X$ over $K$ to the set of morphisms $X\to Z$ over $K$.

\begin{Theorem}[{\cite[Theorem 8.7]{etale-cohomology-of-diamonds}}, {\cite[Theorem 10.1.5, Proposition 10.2.3]{ScholzeBerkeleyLectureNotes}}]\label{l:ff-embedding-into-v-sheaves}\leavevmode
	\begin{enumerate}
		\item The presheaf $Z^{\diamondsuit}$ is already a $v$-sheaf. 
		\item Assume that $K$ is of characteristic $0$ and let $Z$ be a semi-normal (e.g.\ smooth) rigid space $Z$ over $K$. Then sending $Z\mapsto Z^\diamondsuit$ defines a fully faithful functor
		\[\{ \text{semi-normal rigid spaces over } K \}\hookrightarrow\{
		v\text{-sheaves on }\Perf_{K} \}. \]
	\end{enumerate}
\end{Theorem}

This point of view is crucial to our considerations: It implies that when studying smooth adic $K$-groups and their perfectoid covers, we can regard these as abelian sheaves on $\Perf_K$ in a faithful way. This is desirable because in contrast to the category of adic groups, the category of abelian sheaves on $\Perf_K$ is abelian, and allows us to employ homological algebra. 

In contrast, in characteristic $p$, going from rigid spaces to $v$-sheaves is far from lossless:

\begin{Definition}[{\cite[Definition III.2.18.(ii)]{torsion}}]\label{d:perf}
	Let $G$ be a  rigid space over $K$ of characteristic $p$. Then there is a unique perfectoid tilde-limit $G^{\perf}\sim\varprojlim_{F}G$ where $F$ is the absolute Frobenius. The natural map $G^{\perf}\to G$ induces an isomorphism of $v$-sheaves.
\end{Definition}
From now on, we will freely identify analytic adic spaces over $K$ with their associated $v$-sheaves, i.e.\ write $Z$ for what we denoted by $Z^\diamondsuit$ above.

\begin{Definition}[{\cite[Definition 3.1.1]{bhatt-scholze-proetale}}]\label{d:replete-BS}
	A site $\mathcal C$ is replete if for any inverse system $(F_n)_{n\in\N}$ of sheaves on $\mathcal C$ with surjective transition maps, each projection $\varprojlim_{n\in \N} F_n\to F_m$ is surjective.
\end{Definition}
\begin{Lemma}[{\cite[{Proposition~3.1.10}]{bhatt-scholze-proetale}, \cite[Lemma~2.6]{heuer-Picard-good-reduction}}]\label{l:qproet-replete}
	Let $X$ be a diamond. Then $X_{\qproet}$ and $X_{v}$ are replete. In particular, on either site, for $(F_n)_{n\in\N}$ as in Definition~\ref{d:replete-BS} we have
\[ \textstyle\Rlim_nF_n=\varprojlim_nF_n.\]
\end{Lemma}

\subsection{The topologically $p$-torsion subsheaf}
To any locally profinite space $S$, one can associate a perfectoid space $\underline{S}$, and in particular a $v$-sheaf, whose underlying topological space is $S$. If $S=\varprojlim S_i$ is profinite, this is done by
\[ \underline{S}:=\varprojlim \underline{S_i}\]
where $\underline{S_i}$ is the constant sheaf. We refer to \cite[Example 11.12]{etale-cohomology-of-diamonds}\cite[\S2]{heuer-diamantine-Picard} for more details.
\begin{notation}
	For typesetting convenience, we will write $\uZp$ (rather than $\underline{\Z_p}$) for the profinite sheaf associated to $\Z_p$. This is justified by the canonical isomorphism $\uZp=\varprojlim \uZ/p^n$.
\end{notation}
\begin{Lemma}[{\cite[Lemma~2.4]{heuer-diamantine-Picard}}]\label{l:underline-is-left-adjoint}
	Let $H$ be a locally profinite set, and let $G$ be any adic space over $(K,K^+)$. Then evaluation on points defines an isomorphism, functorial in $H$ and $G$:
	\[ H^0(\underline{H},G)=\Mapc(H,G(K,K^+)).\]
\end{Lemma}
Let $G$ be a rigid or perfectoid $K$-group, then by \cite[Proposition~A.6]{heuer-diamantine-Picard}, the $K$-points $G(K):=G(K,\O_K)\subseteq |G|$ inherit the structure of a complete Hausdorff topological group.
\begin{Proposition}[{\cite[Proposition~2.15.2]{heuer-diamantine-Picard}}]\label{p:morphisms-from-Z,Z_p,Z_p/Z}
	For $G$ a rigid or perfectoid group over $K$,  
	\begin{alignat*}{3}
		\Hom(&\uZ,&G)&=&&\Hom(\Z,G(K))=G(K),\\
		\Hom(&\uZp,&G)&=&&\Hom_{\cts}(\Z_p,G(K))=\{x\in G(K)\mid p^n\Z\cdot x\xrightarrow{n\to \infty} 0\}.
	\end{alignat*}
\end{Proposition}
This motivates the following definition, see \cite[Definition~2.16]{heuer-diamantine-Picard} for more details:
\begin{Definition}\label{d:G-hat}
	For any adic group $G$, we denote by $\widehat{G}\to G$ the image of the evaluation map $\HOM(\uZp,G)\to G$ at $1\in \Z_p$. We call this the topologically $p$-torsion subsheaf.
\end{Definition}
	In all cases that we will encounter in the following, $\widehat{G}\subseteq G$ will be represented by an adic subgroup. For example, if $\Char K=0$ and $G$ is a rigid group over $K$, then $\widehat{G}$ may be identified with an open subgroup considered by Fargues in the context of analytic $p$-divisible groups \cite[\S1.6 Th\'eor\'eme~1.2]{Fargues-groupes-analytiques}. See \cite[\S2]{heuer-diamantine-Picard} for details. We note the following special case:
\begin{Lemma}[{\cite[Example~2.23.1]{heuer-diamantine-Picard}}]\label{l:K-points-of-whB}
	Let $B$ be an abeloid variety of good reduction over $K$ with special fibre $\overbar{B}$ over $k$. Then the $K$-points of the topological torsion group fit into  an exact sequence
	\[ 0\to \wh B(K)\to B(K)\to \overbar{B}(k)\tf\to 0\]
\end{Lemma}
This sequence in fact has a geometric incarnation as a short exact sequence of $v$-sheaves that naturally came up in \cite[\S5]{heuer-Picard-good-reduction} in the context of Picard functors of universal covers:
\begin{Definition}\label{d:diamond-on-special-fibre}
	To any $k$-scheme $X$, we associate the $v$-sheaf $X^\diamond$ defined by sheafification of the functor on $\Perf_K$
	\[ (S,S^+)\mapsto X(S^+/\mathfrak m),\]
	where $S^+/\mathfrak m$ is considered as a $k$-algebra (recall that $\mathfrak m\subseteq \O_K$ is the maximal ideal). In fact, by \cite[Lemma~5.2]{heuer-Picard-good-reduction} it suffices to form the analytic sheafification, which is already a $v$-sheaf.
\end{Definition}
\begin{Proposition}[{\cite[Corollary~5.6]{heuer-Picard-good-reduction}}]\label{p:ses-whB-B-barBtf}
	There is a short exact sequence on $\Perf_{K,v}$
	\[ 0\to \wh{B}\to B\to \overbar{B}^\diamondsuit\tf\to 0.\]
\end{Proposition}
This is another way to say that $\wh B$ is associated to the $p$-divisible group of $B$. It can be used to give a description of $\wh{B}$ in terms of an analytic Weil pairing that will be useful:
\begin{Proposition}[{\cite[\S4]{tate1967p}, \cite[Theorem~5.1]{heuer-diamantine-Picard}}]\label{p:analytic-Weil}
	Assume $\Char K=0$ and let $A$ be an abeloid of good reduction over $K$. Then the Weil pairing  induces a split exact sequence of $v$-sheaves
	\[ 0\to \wh A\to \HOM(T_pA,\wh \G_m)\to H^0(A,\Omega^1(-1))\otimes_K \G_a\to 0.\]
	In particular, any choice of basis $\Z_p^{2d}\cong T_pA$ induces a split injection $\wh A\hookrightarrow \wh \G_m^{2d}$.
\end{Proposition}
\begin{proof}
	This follows from the cited Theorem by passing to topological $p$-torsion subsheaves, thus killing prime-to-$p$-torsion. We note that if $A$ has good reduction, the first map can easily be constructed explicitly on the level of formal schemes, see for example \cite[\S4]{tate1967p}.
\end{proof}

In contrast to $p$-torsion, for any prime $l\neq p$, topological $l$-torsion is the same as $l$-torsion:
\begin{Lemma}[{\cite[Lemma~2.24]{heuer-diamantine-Picard}}]\label{l:morphism-from-pro-coprime-to-p}
	Let $H=\varprojlim H_i$ be a profinite group for which the order of each $H_i$ is coprime to $p$. Let $G$ be a quasi-compact smooth rigid group. Then
	\[\Hom(\underline{H},G)=\textstyle\varinjlim_{i}\Hom(H_i,G(K)).\]
\end{Lemma}

\subsection[The quotient sheaf Zp/Z]{The quotient sheaf \texorpdfstring{$\uZp/\uZ$}{Zp/Z}}

Motivated by \eqref{eq:intro-E-A-Mp/M} in the introduction, we will be interested in the sequence of abelian sheaves
\begin{equation}\label{eq:ses-Z-Zp-Zp/Z}
0\to \underline{\Z}\to \uZp\to \uZp/\underline{\Z}\to 0
\end{equation}
on $\Perf_{K,v}$
where the right term is defined as the cokernel. We should a priori be careful which topology we work in (analytic, quasi-pro-\'etale, $v$), but in fact this is no problem:
\begin{Lemma}\label{l:uZp/uZ on an vs v} 
	Let $\lambda:\Perf_{K,v}\to \Perf_{K,\an}$ be the natural morphism of sites. Then
	\[\lambda^{\ast}(\uZp/\uZ)=\uZp/\uZ, \quad \lambda_{\ast}(\uZp/\uZ)=\uZp/\uZ.\]
\end{Lemma}
\begin{proof}
	The sheaves $\uZ$ and $\uZp$ are both represented by perfectoid spaces, and thus $\lambda^{\ast}\uZ=\uZ$ and $\lambda^{\ast}\uZp=\uZp$. The first part therefore follows from exactness of $\lambda^{\ast}$.
	
	The second part follows from  $\lambda_{\ast}\uZ=\uZ$, $\lambda_{\ast}\uZp=\uZp$ and $R^1\lambda_{\ast}\uZ=0$, which holds by Proposition~\ref{p:Z-torsors-over-perfectoids} below. The proof is not difficult but it fits more naturally into \S\ref{s:Z-torsors}
\end{proof}
\begin{Lemma}\label{l:H^0(Z_p/Z,G)}
	Let $G$ be an adic space  over $K$. Then any morphism $\uZp/\underline{\Z}\to G$ is constant.
\end{Lemma}
\begin{proof}
	Any such function is determined by its pullback to $\uZp$, and is constant upon pullback to $\underline \Z$. But using Lemma~\ref{l:underline-is-left-adjoint},  we see that the pullback map
	\[H^0(\uZp,G)=\Map_{\cts}(\Z_p,G(K))\to H^0(\underline{\Z},G)=\Map(\Z,G(K)),\]
	is injective. Thus $f$ is already constant.
\end{proof}

\begin{Lemma}\label{l:isoms-of-Zp/Z}
	\begin{enumerate}
		\item\label{enum:uzp/uz-p-div} The sheaf $\uZp/\uZ$ is uniquely $p$-divisible (in the sense of Definition~\ref{d:p-div-sheaf}).
		\item\label{enum:uzp/uz-p=uQp/Ztf} The natural map $\uZp/\uZ\to \underline{\Q_p}/\underline{\Z[\tfrac{1}{p}]}$ is an isomorphism.
		\item\label{enum:points-of-uZp/uZ} We have $(\uZp/\uZ)(K)=\Z_p/\Z$.
		\item\label{enum:ext 1(uzp,z)} 	We have  $\Ext^1_v(\uZp,\underline{\Z})=\Q_p/\Z_p$.
		\item \label{enum:uZp-uZp/uZ-lifts}
		Every homomorphism $\varphi:\uZp\to \uZp/\uZ$ lifts uniquely to a $\uZp$-linear map $\tilde{\varphi}:\uZp\to \uQp$.
		\item\label{enum:End(uZp/uZ)} We have $\End(\uZp/\uZ)=\Z[\tfrac{1}{p}]$.
	\end{enumerate}
\end{Lemma}
\begin{proof}
	\begin{enumerate}
		\item By Lemma~\ref{l:qproet-replete}, we have $\uZp/p=\underline{\Z/p}$, so this follows from the diagram
		\begin{center}
			\begin{tikzcd}
			0 \arrow[r] & \underline{\Z} \arrow[d,"\cdot p"] \arrow[r] & \uZp \arrow[d,"\cdot p"] \arrow[r] & \underline{\Z}_p/\underline{\Z} \arrow[d,"\cdot p"] \arrow[r] & 0 \\
			0 \arrow[r] & \underline{\Z} \arrow[r] \arrow[d] & \uZp \arrow[r] \arrow[d] & \uZp/\underline{\Z} \arrow[r]           & 0 \\
			& \underline{\Z/p} \arrow[r,equal]       & \underline{\Z/p}.                   &                                                     &  
			\end{tikzcd}
		\end{center}
		\item This follows from 1 by applying $\varinjlim_{\cdot p}$ to the short exact sequence \eqref{eq:ses-Z-Zp-Zp/Z}.
		\item This follows from $H^1_v(K,\Z)=0$ which holds e.g.\ by Proposition~\ref{p:higher-cohomolgy-of-profinite=0}.\ref{enum:H^i(profinite,profinite)}.
		\item By Proposition~\ref{p:higher-cohomolgy-of-profinite=0}.\ref{enum:Ext(constant,profinite)}, we have
		\[\Ext^1_v(\uZp,\uZ)=\varinjlim_n \Ext^1_{\Z}(\Z/p^n,\Z)=\varinjlim_n\Z/p^n=\Q_p/\Z_p.\]
		\item This follows from \ref{enum:ext 1(uzp,z)} and the long exact sequence
		\[ \Hom(\uZp,\uZ)=0\to \Hom(\uZp,\uZp)=\Z_p\to \Hom(\uZp,\uZp/\uZ)\to \Ext^1_v(\uZp,\uZ)=\Q_p/\Z_p, 
		\]		
		which shows that after multiplying $\varphi$ by $p^n$ for $n\gg 0$, it lifts to an endomorphism $p^n\tilde{\varphi}:\uZp\to \uZp$. Using 2, this shows that after dividing by $p^n$, $\tilde{\varphi}$ defines a lift to $\uQp$.
		\item Applying $\Hom(-,\uZp/\uZ)$ to \eqref{eq:ses-Z-Zp-Zp/Z}, we get by \ref{enum:points-of-uZp/uZ}, \ref{enum:uZp-uZp/uZ-lifts} and Proposition~\ref{p:morphisms-from-Z,Z_p,Z_p/Z} an exact sequence 
		\[ 0\to \Hom(\uZp/\uZ,\uZp/\uZ)\to \Hom(\uZp,\uZp/\uZ)=\Q_p\to \Hom(\uZ,\uZp/\uZ)=\Z_p/\Z.\]
		The kernel is precisely $\Z[\tfrac{1}{p}]$ by the argument in \ref{enum:uzp/uz-p=uQp/Ztf}.
		\qedhere
		.	\end{enumerate}
\end{proof}

\section{Universal covers of abelian varieties as $v$-sheaves}
In this section, we collect some results on universal covers of abeloid varieties that we need. In particular, we review some results from \cite{heuer-diamantine-Picard,heuer-Picard-good-reduction,heuer-v_lb_rigid} about line bundles on universal covers.
\begin{Definition}
	Let $G$ be any abelian $v$-sheaf. We write
	\[ \widetilde{G}:=\textstyle\varprojlim_{[p]} G,\quad 
	\wtt G:=\textstyle\varprojlim_{[N]}G\]
	where $N$ ranges over $N\in\N$. We set $T_pG:=\varprojlim_{[p]}G[p^n]$. There is then a left-exact sequence
	\[ 0\to T_pG\to \widetilde{G}\to G\]
	on $\Perf_{K,v}$.
	By Lemma~\ref{l:qproet-replete}, this is right-exact if and only if $G$ is $p$-divisible.
\end{Definition}
\begin{Remark}
In our applications, $G$ is a smooth rigid group over $K$ for which $\wt G$ is represented  by a perfectoid space. In this case, if the map $[N]:G\to G$ is finite \'etale for all $N$ coprime to $p$, then $\wtt G$ is also represented by a perfectoid space, and $\wtt G\to \wt G$ is pro-\'etale.
\end{Remark}
\begin{Example}\label{ex:G_m}
	Let $G=\G_m$. Then there is a perfectoid tilde-limit
	\[ \wt \G_m\sim \textstyle\varprojlim_{[p]}\G_m.\]
	If $\Char K=p$, this is the perfection of $\G_m$, which represents the same $v$-sheaf as $\G_m$. If $\Char K=0$, it is the untilt of the former, and we have an exact sequence on $\Perf_{K,v}$
	\[0\to \Z_p(1):=T_p\G_m\to \widetilde{\G}_m\to \G_m\to 0.\]
\end{Example}
\begin{Example}
	We have $\widetilde{\underline{\Z}}=0=\wt{\uZ}_p$, and $\wt{\underline{\Q_p/\Z_p}}=\underline{\Q_p}$.
\end{Example}
\begin{Remark}
	The sheaf $\wt G$ is what is called the  ``universal cover'' of $G$ in \cite[Definition 13.1]{ScholzePSandApplications}, where it originates from the context of $p$-divisible group. If $G$ is an abeloid, then $\wtt G$ is the universal pro-finite-\'etale cover in the sense of \cite[\S4.3, Example~4.7.2]{heuer-v_lb_rigid}. We will see in \S\ref{s:top-up} in what topological sense $\wt G$ and $\wtt G$ are ``universal covers'' for an abeloid variety $G$.
\end{Remark}
\subsection{Recollections on universal covers of abeloid varieties}
Let $A$ be an abeloid variety over $K$, i.e.\ a connected smooth proper rigid group over $K$. Then by the theory of Raynaud extensions, one can associate to $A$ a short exact sequence
\begin{equation}\label{eq:Raynaud-unif-1}
0\to T\to E\to B\to 0
\end{equation}
of analytic adic spaces over $K$, where $B$ is an abelian variety with good reduction and $T$ is a rigid torus (necessarily split) of some rank $r$. This sequence is locally split in the analytic topology on $B$, and thus exact in the analytic topology. Moreover, there is a lattice $\underline{\Z}^r\cong M\subseteq E$ such that there is a short exact sequence in the analytic topology
\begin{equation}\label{eq:Raynaud-unif-2}
0\to M\to E\to A\to 0.
\end{equation}
The data of these two sequences, in the following abbreviated to ``$A=E/M$'', is commonly referred to as the Raynaud uniformisation of $A$. 
Translated into the language of $v$-sheaves, by Theorem~\ref{l:ff-embedding-into-v-sheaves}, we  regard \eqref{eq:Raynaud-unif-1} and \eqref{eq:Raynaud-unif-2} as exact sequences of abelian sheaves on $\Perf_{K,\an}$.

We now review the main results of \cite{perfectoid-covers-Arizona} on universal covers of abeloid varieties and their Raynaud extensions, translated into the category of sheaves on $\Perf_{K,v}$.

 \begin{Theorem}[{\cite[Proposition~3.8, Theorem~4.6]{perfectoid-covers-Arizona}}]\label{t:perfectoid-covers-of-abelian-varieties} Let $K$ be an algebraically closed non-archimedean field.
 	Let $A$ be an abeloid over $K$ with Raynaud uniformisation $A=E/M$.
 	\begin{enumerate}
 	\item The universal covers $\widetilde{T}$, $\widetilde{E}$, $\widetilde{B}$ of $T$, $E$, $B$ are all represented by perfectoid groups. They fit into a short exact sequence of abelian sheaves on $\Perf_{K,\an}$
 	\[0\to \widetilde{T}\to \widetilde{E}\to \wt B\to 0. \]
 	\item The universal cover $\widetilde{A}=\varprojlim_{[p]}A$ is also represented by a perfectoid group. It fits into a short exact sequence on $\Perf_{K,v}$
 	\[0\to T_pA\to \wt A\to A\to 0. \]
 	\item Let $M_p:=M\otimes_{\uZ}\uZp=\varprojlim M/p^nM\cong \uZp^r$. Then any choice of lift of the inclusion $g:M\to E$ to $\tilde{g}:M\to \wt E$ induces a short exact sequence on $\Perf_{K,v}$
 	\[0\to \wt E \to \widetilde{A} \to M_p/M\to 0.\]
\end{enumerate}
\end{Theorem}
\begin{proof}
	We only need to explain the last part, as this is not explicitly stated in the cited article. By \cite[Theorem 6.4]{perfectoid-covers-Arizona}, there is a perfectoid space $A_\infty$ which fits into a pushout diagram of short exact sequences on $\Perf_{K,v}$
 	\begin{equation}\label{eq:wtE-wtA-vs-Ainfty}
 		\begin{tikzcd}
 		0 \arrow[r] & M \arrow[r,"\wt g"] \arrow[d, hook] & \wt E \arrow[d, hook] \arrow[r] & A_\infty \arrow[d,equal] \arrow[r] & 0 \\
 		0 \arrow[r] & M_p \arrow[r] & \widetilde{A} \arrow[r] & A_\infty\arrow[r] & 0.
 		\end{tikzcd}
 \end{equation}
 	The desired statement follows by forming cokernels.
 \end{proof}
\begin{Corollary}\label{c:interpret-of-boundary-map-for-wt-E}
	The homological boundary maps of $0\to M\to M_p\to M_p/M\to 0$,
		\begin{center}
		\begin{tikzcd}
		\Hom(M,\wt E) \arrow[r] \arrow[d] &\Ext^1_v(M_p/M,\wt E) \arrow[d, equal]\\
		\Hom(M,E) \arrow[r] & \Ext^1_v(M_p/M,E),
		\end{tikzcd}
	\end{center}
	send $g$ (respectively, any lift $\tilde{g}$ of $g$) to the extension class of $[\wt A]\in \Ext^1(M_p/M,\wt E)$.
\end{Corollary}
\begin{proof}
	Immediate from diagram~\eqref{eq:wtE-wtA-vs-Ainfty} and Lemma~\ref{l:interpretation-of-Ext1-as-extensions}. The isomorphism on the right comes from the long exact sequence of $T_pE\to \wt E\to E$ using Lemma~\ref{l:isoms-of-Zp/Z}.1 and Lemma~\ref{l:ext-from-p-divisible-to-complete}.
\end{proof}

\subsection{The topological universal property of the universal cover}\label{s:top-up}

We now recall  the topological universal property of $\wt A$:
\begin{Proposition}[{\cite[Proposition~4.2]{heuer-Picard-good-reduction}}]\label{p:H^1_v(A,Zp)=0}
	Let $A$ be an abeloid variety over $K$. Then 
	\[H^i_v(\wt A,\uZp)=H^i_{\qproet}(\wt A,\uZp)=\begin{cases}\Z_p&\quad \text{ for }i=0\\0& \quad \text{ for }i>0,\end{cases}\]
	\[H^i_{v}(\wt A,\O^+)\aeq H^i_{\qproet}(\wt A,\O^+)\aeq H^i_{\an}(\wt A,\O^+)\aeq \begin{cases}\O_K&\quad \text{ for }i=0\\0& \quad \text{ for }i>0.\end{cases}\]
	The same applies to $\wtt A$ with $\uZp$ replaced by $\underline{\whZ}$.
\end{Proposition}

\begin{Remark}\label{r:H^1(wt A,Z)}
	In contrast, we will see in Proposition~\ref{p:morphis-wtT-to-B} that $H^1_v(\wt A,\uZ)=\Hom(M,\Z)\tf$.
\end{Remark}
\begin{Corollary}\label{c:cor-to-top-univ-property}
	The cover $\wt A\to A$ has the following topological universal property:
		Let $X$ be another diamond with $H^1_{\qproet}(X,\uZp)=0$. Then any map $\varphi:X\to A$ lifts to a map
		\begin{center}
			\begin{tikzcd}
			& \wt A \arrow[d] \\
			X \arrow[r,"\varphi"] \arrow[ru, "\exists \,\tilde{\varphi}", dotted] & A.   
			\end{tikzcd}
		\end{center}
		Moreover, if $X$ is connected and if there is $x\in X(K)$ such that $\varphi(x)=0\in A(K)$, then there is a unique lift $\tilde{\varphi}$ for which $\tilde{\varphi}(x)=0$.
\end{Corollary}
\begin{proof}[Proof of Corollary~\ref{c:cor-to-top-univ-property}]
	The first statement follows from the exact sequence
		\[ H^0(X,T_pA)\to H^0(X,\wt A)\to H^0(X,A)\to H^1_{\qproet}(X,T_pA)=0.\]
		
		To see the uniqueness assertion,  we note that exactness of the sequence means that any two lifts $\tilde{\varphi}_1,\tilde{\varphi}_2$ agree up to a morphism $X\to T_pA$. 
		If $X$ is connected, this means that
		$H^0(X,T_pA)=T_pA$ acts simply transitively  on the points of $\wt A(K)$ over $0\in A(K)$.
\end{proof}
\begin{Corollary}\label{c:H^1(wt A,wh A')}
	We have $H^i(\wt A,\wh\G_m)=0$ for $i> 0$. If $A'$ is another abeloid, $H^i(\wt A,\wh A')=0$.
\end{Corollary}
\begin{proof}
	The first part follows from Theorem~\ref{p:H^1_v(A,Zp)=0} using the logarithm short exact sequences like in \cite[Proposition~4.6]{heuer-diamantine-Picard}. The second part follows from this using
	Proposition~\ref{p:analytic-Weil}.
\end{proof}

By Lemma~\ref{l:E_2-of-discrete-torsor-over-connected} and the long exact sequence of Theorem~\ref{t:perfectoid-covers-of-abelian-varieties} we deduce:
\begin{Corollary}
	We have  $\Ext^1_v(\wt A,\uZ_p)=0$ and $\Ext^1_{v}(A,\uZp)=\Hom(T_pA,\uZp)$. In particular, $T_pA\to \wt A\to A$
	is the universal extension of $A$ by a finite free $\uZp$-module.
\end{Corollary}

The proof of Proposition~\ref{p:H^1_v(A,Zp)=0} relies on the Primitive Comparison Theorem if $\Char K=0$, and the following tilting result if $\Char K=p$, which will also be useful for us in the following:
\begin{Theorem}[{\cite{projectwithPeterWear},\cite[\S5]{wear2020perfectoid}}]\label{t:tilting-abeloids}
	For every abeloid $A$ over $K$ there is an abeloid $A'$ over $K^\flat$ such that $\wt A^\flat\cong \wt A'$, and vice versa. We  have $\wt E^\flat\cong \wt E'$ for the associated Raynaud extensions.
\end{Theorem}

\subsection{The Picard group of the universal cover}\label{s:Pic}
Finally, we collect some results on the Picard group of $\wt A$, which we studied in \cite{heuer-Picard-good-reduction}. 
\begin{Definition}
	Let $\tau$ be one of $v$ or $\et$. A $\tau$-line bundle on a diamond $X$ is by definition a $\G_m$-torsor on $X$ in the $\tau$-topology. We refer to \cite[\S2]{heuer-v_lb_rigid} for some background on line bundles on diamonds. The group of $\tau$-line bundles up to isomorphism is the $\tau$-Picard group
\[ \Pic_\tau(X)=H^1_\tau(X,\G_m).\]
\end{Definition}
If $X$ is perfectoid, the notions of line bundles in these different topologies agree:
\begin{Theorem}[Kedlaya--Liu, {\cite[Theorem 3.5.8]{KedlayaLiu-II}}]\label{t:KL-line-bundles}
	Let $X$ be a perfectoid space, then \[H^1_{\an}(X,\G_m)=H^1_{\et}(X,\G_m)=H^1_{v}(X,\G_m).\]
\end{Theorem} 
\begin{Remark}
	By Theorem~\ref{t:perfectoid-covers-of-abelian-varieties}.2, this applies in particular to the perfectoid space $\wt A$, and we will therefore in the following often drop $\tau$ from notation.
	From Theorem~\ref{t:perfectoid-covers-of-abelian-varieties} and Proposition~\ref{p:H^1_v(A,Zp)=0}, one can deduce using the sequence of Example~\ref{ex:G_m} that
	$\Pic(\wt A)=\Pic(\wt A^\flat)$.
\end{Remark}
The following result from \cite{heuer-Picard-good-reduction} is crucial for us in order to describe isomorphism classes of universal covers $\wt E$ of Raynaud extensions  as extension classes in $\Ext^1_v(\wt B,\wt T)$.
\begin{Theorem}[{\cite[{Theorem~4.1, Corollary~4.13}]{heuer-Picard-good-reduction}}]\label{t:Pic(wt B)}
	Let $B$ be an abeloid variety of good reduction over $K$. Let $\overbar{B}$ be the special fibre of $B$ considered as a scheme over $k$. Then 
	\[ \Pic(\wt B)=\Pic(\overbar{B})[\tfrac{1}{p}].\]
	In particular, let $\overbar{B}^\vee$ be the dual of $\overbar{B}$, then for any torus $T$ with character group $M^\vee$, 
	\[ \Ext^1_{\an}(\wt B, T)=\Hom(M^\vee,\overbar{B}^\vee(k)\tf).\]
\end{Theorem}
Second, we will independently need line bundles to study morphisms between universal covers. These are related to line bundles on universal covers by way of the ``diamantine Picard functor'' from \cite{heuer-diamantine-Picard}, as we shall now discuss. For simplicity, we assume $\Char K=0$.

 We first recall the rigid Picard functor of abeloids, as studied by Bosch--L\"utkebohmert:
\begin{Definition}
	Let $A$ be an abeloid variety and let $X$ be any rigid or perfectoid space over $K$. Then the translation-invariant Picard group $\Pic^\tau(A\times X)$ is defined as the kernel of
	\[\Pic_{\et}(A\times X)\xrightarrow{\pi_1^\ast+\pi_2^\ast-m^\ast} \Pic_{\et}(A\times A\times X),\]
	where $\pi_1,\pi_2,m:A\times A\to A$ are the projection and multiplication maps, respectively.
	Here since $A$ is smooth, we could consider $A\times X$ as a sousperfectoid adic space in the sense of \cite[\S7]{HK_sheafiness}, but for our purposes it will be convenient to instead consider $A\times X$ as a diamond.
\end{Definition}

The translation-invariant Picard functor of $A$ is defined as the \'etale sheafification of the functor $X\to \Pic^\tau(A\times X)$ on rigid spaces $X$ over $K$. By the duality theory of abeloids of Bosch--L\"utkebohmert \cite[\S6]{BL-Degenerating-AV}, this is represented by an abeloid $A^\vee$ over $K$, called the dual abeloid. As usual, the universal line bundle on $A\times A^\vee$ induces an isomorphism $(A^\vee)^\vee=A$.

In \cite[Theorem~1.1.1]{heuer-diamantine-Picard}, we had shown that Picard functors defined on perfectoid test objects are represented by the diamondification  of the rigid Picard functor. In particular:
\begin{Theorem}[{\cite[Corollary~3.44]{heuer-diamantine-Picard}}]\label{t:Pict}
	Assume that $\Char K=0$. Let $A$ be an abeloid variety over $K$. Let $A^\vee$ be the dual abeloid of $A$. Let $X$ be any perfectoid space over $K$. Then specialisation at $0\in A(K)$ defines a natural isomorphism
	\[\Pic^{\tau}(A\times X)=  A^\vee(X)\times \Pic(X).\]
\end{Theorem}
Finally, there is a version of this for the universal cover which geometrises Theorem~\ref{t:Pic(wt B)}. 
This will later be useful to  see that all morphisms $\wt T\to \wt B$ are constant. On $\wt B$, line bundles turn out to be essentially the same as torsors under the following quotient sheaf:
\begin{Definition}[{\cite[Definition~2.12]{heuer-v_lb_rigid}}]\label{d:bOx}
	The $v$-sheaf $\bOx$ is defined as the quotient on $\Perf_{K,v}$
	\[ \bOx:=\O^\times/(1+\m\O^+).\]
\end{Definition}
The following result from \cite{heuer-Picard-good-reduction} makes precise the idea that the Picard functor of $\wt B$ is represented by the Picard scheme of the special fibre $\overbar{B}$, using the sheaf of Definition~\ref{d:diamond-on-special-fibre}:
\begin{Theorem}[{\cite[Corollary~5.6]{heuer-Picard-good-reduction}}]\label{t:Pic-functor-wtB}
	Assume that $\Char K=0$. Let $B$ be an abeloid over $K$ of good reduction $\overbar{B}$ over $k$. Let $\uP_{\overbar B}$ be the Picard scheme of $\overbar{B}$ over $k$. Let $X$ be any perfectoid space over $K$. Then specialisation at $0\in \wt B(K)$ defines natural isomorphisms
	\begin{alignat*}{3}
	\Pic(\wt B\times X)&=&& \uP_{\overbar B}^\diamond\tf(X)\times \Pic(X),\\
	H^1_{\et}(\wt B\times X,\bOx)&=&& \uP_{\overbar B}^\diamond\tf(X)\times H^1_{\et}(X,\bOx).
	\end{alignat*}
\end{Theorem}
\section{Morphisms of universal covers: Raynaud extensions}
With these preparations, we now return to our initial question:
\begin{Question}\label{q:hom-of-covers}
	For two abeloids $A$, $A'$, what are all morphisms $\widetilde{A}\to \widetilde{A}'$ of perfectoid spaces?
\end{Question}
Of course, by functoriality of the universal cover, any homomorphism $\varphi:A\to A'$ gives rise to a morphism $\widetilde{\varphi}:\widetilde{A} \to\widetilde{A}'$ which is an isomorphism if $\varphi$ is an isogeny of $p$-power degree. In this section, we will see that there are often morphisms that do not arise in this way.

\subsection{Case of good reduction}
Our first step is to give a precise answer to Question~\ref{q:hom-of-covers} in the case of good reduction:

\begin{Theorem}\label{t:isomorphisms-between-ab-var-of-good-reduction}
	Let $B,B'$ be two abeloid varieties of good reduction over $K$. Let $\overbar B$ and $\overbar B'$ be the special fibres over the residue field $k=\O_K/\m$. Then:
	\begin{enumerate}
	 \item There is a natural isomorphism, compatible with composition:
	\[\Hom(\wt B,\wt B')=\Hom(\overbar B,\overbar B')[\tfrac{1}{p}].\]
	\item Any morphism $\wt B\to \wt B'$ is the composition of a homomorphism and a translation.
	\item In particular, the following are equivalent:
	\begin{enumerate}
		\item There is an isomorphism $\widetilde{B}\to \widetilde{B}'$ of perfectoid groups over $K$.
		\item There is an isomorphism $\widetilde{B}\to \widetilde{B}'$ of perfectoid spaces over $K$.
		\item  There is an isogeny $\overbar{B}\to \overbar{B}'$ of  $p$-power degree over $k$.
	\end{enumerate}
	If $\Char K=p$, then $\Hom(\wt B,\wt B')=\Hom(B^{\perf}, B'^{\perf})[\tfrac{1}{p}]$, and these are equivalent to:
	\begin{enumerate}
		\setcounter{enumii}{3}
		\item  There is an isogeny $B^{\perf}\to B'^{\perf}$ of $p$-power degree over $K$, i.e.\ a surjection of $v$-sheaves whose kernel is a finite \'etale group scheme over $K$ of $p$-power order.
	\end{enumerate}
	\end{enumerate}
\end{Theorem}
That $(d)\Rightarrow (a)\Rightarrow (b)$ is clear.
The direction $c)\Rightarrow a)$ is already contained in \cite[Proposition 13.2]{ScholzePSandApplications}. Indeed, $c)\Rightarrow a)$ and $c)\Rightarrow d)$ easily follow from the tilting equivalence \cite[Theorem~5.2]{perfectoid-spaces}, combined with the following two standard  lemmas on schemes:
\begin{Lemma}[{\cite[Proposition 8.13.1]{EGA-IV-1}}]\label{l:morph-proj-lim-schemes-to-ft-scheme}
	Let $S$ be a scheme.
	Let $(X_i)_{i\in I}$ be a cofiltered inverse system of qcqs $S$-schemes with affine transition maps. Let $Y\to S$ be of finite presentation. Then
	\[\Mor_S(\varprojlim_{i\in I}X_i,Y)=\varinjlim_{i \in I} \Mor_S(X_i,Y). \]
\end{Lemma}
\begin{Lemma}\label{l:lifting-from-mod-m}
	Let $X,Y$ be two qcqs schemes over $\O_K$ and assume that  $Y$ is of finite presentation over $\O_K$. Then any morphism $f:X\times_{\O_K}k\to Y\times_{\O_K}k$ can be lifted to a morphism \[\tilde{f}:X\times_{\O_K}\O_K/\varpi^\epsilon\to Y\times_{\O_K}\O_K/\varpi^\epsilon\]
	for some $\epsilon>0$. If $X$ and $Y$ are group schemes and $f$ is a morphism of group schemes, one can arrange for $\tilde{f}$ to be a morphism of group schemes.
\end{Lemma}
\begin{proof}
	This follows from Lemma~\ref{l:morph-proj-lim-schemes-to-ft-scheme}: We have $k=\varinjlim_{\epsilon\to 0} \O_K/\varpi^{\epsilon}$. Consequently,
	\[\Mor_{\O_K}(X\times_{\O_K}k,Y)= \Mor_{\O_K}(\varprojlim_{\epsilon\to 0}(X\times_{\O_K}\O_K/\varpi^{\epsilon}),Y)=\varinjlim_{\epsilon\to 0}\Mor_{\O_K}(X\times_{\O_K}\O_K/\varpi^{\epsilon},Y).\]
	The other statements follow from this.
\end{proof}
\begin{proof}[Proof of Theorem~\ref{t:isomorphisms-between-ab-var-of-good-reduction}]
We first prove the direction $c)\Rightarrow a)$: 

Let $\mathfrak B$, $\mathfrak B'$ be the formal models of $B$ and $B'$ over $\O_K$, then by Lemma~\ref{l:lifting-from-mod-m}, any homomorphisms $\overbar{\varphi}:\overbar{B}\to \overbar{B}'$ lifts to a homomorphism $\mathfrak B/\varpi^\epsilon\to \mathfrak B'/\varpi^\epsilon$ for some $\epsilon>0$. Passing to the universal covers $\wt{\mathfrak B}:=\varprojlim_{[p]} \mathfrak B$ and $\wt{\mathfrak B}':=\varprojlim_{[p]} \mathfrak B'$, this induces by applying $\varprojlim_{[p]}$ a homomorphism
\[ \wt{\mathfrak B}/\varpi^\epsilon\to \wt{\mathfrak B}'/\varpi^\epsilon\]
that is an isomorphism if $\overbar{\varphi}$ is a $p$-isogeny (by considering the dual isogeny).
Since $\wt{\mathfrak B}$ and $\wt{\mathfrak B}'$ are perfectoid over $\O_K$, the tilting equivalence implies that this lifts to an isomorphism in the almost category
\[ \wt{\mathfrak B}\to \wt{\mathfrak B}',\]
and we obtain the desired result on the generic fibre.

The direction  $b)\Rightarrow c)$ is more subtle: The above steps are reversible, except for the very last: It is not obvious that every morphism $\wt B\to \wt B'$ has a formal model $\wt{\mathfrak B}\to \wt{\mathfrak B}'$. To see this, we now translate this into a question about line bundles on universal covers. These are well-understood by the tools developed in the series \cite{heuer-diamantine-Picard,heuer-Picard-good-reduction,heuer-v_lb_rigid}, as summarised in \S\ref{s:Pic}.

Via untilting of abeloids, Theorem~\ref{t:tilting-abeloids}, we can reduce to the case that $\Char K=0$.
	By Proposition~\ref{p:H^1_v(A,Zp)=0} and Corollary~\ref{c:cor-to-top-univ-property}, we have a short exact sequence
	\[0\to T_pB'\to \wt B'(\wt B)\to B'(\wt B)\to 0 \]
	so it suffices to describe the morphisms $\wt B\to B'$.  

In the first part of the proof, we have already seen that there is a natural map
\[ \Hom(\overbar B,\overbar B')\tf\times B'(K)\to B'(\wt B)\]
where the first factor goes to homomorphisms $\Hom(\wt B,\wt B')$ and the second is the translation on the target. We need to see that this is an isomorphism, then all desired statements follow.

Let $B'^\vee$ be the dual abeloid of $B'$.
By Theorem~\ref{t:Pict}, any morphism $\wt B\to B'$ defines a $B'^\vee$-invariant line bundle on $B'^\vee\times\wt B$. More precisely, using the natural isomorphism $(B'^\vee)^\vee=B'$, the theorem says that we have a natural isomorphism
\[  \Pic^{\tau}(B'^\vee\times \wt B)=B'(\wt B)\times \Pic(\wt B).\]
By Theorem~\ref{t:Pic(wt B)}, we have  $\Pic(\wt B)=\Pic(\overbar{B})\tf$. On the other hand,
the group on the left is computed by the following lemma (set $B_1:=B'^\vee$, $B_2:=B$):
\begin{Lemma}\label{l:Pic(wtBxB')}
	Assume that $\Char K=0$ and let $B_1$, $B_2$ be abeloids of good reduction over $K$.  Then pullback along $\wt B_1\times \wt B_2\to B_1\times\wt B_2$ defines a left-exact sequence
	\[ 0\to \wh{B}_1^\vee(K) \to \Pic_{\et}(B_1\times\wt B_2)\to \Pic_{\et}(\wt B_1\times \wt B_2).\]
\end{Lemma}
\begin{proof}
	Since $\O(\wt B_1\times \wt B_2)=K$, the Cartan--Leray sequence (see \cite[Proposition~2.8]{heuer-v_lb_rigid}) for the $T_pB_1$-torsors $\wt B_1\to B_1$ and $\wt B_1\times \wt B_2\to B_1\times \wt B_2$ yields exact sequences, related via pullback,
	\[
	\begin{tikzcd}	
		0\arrow[r]&{\Hom_{\cts}(T_pB_1,K^\times)} \arrow[r] \arrow[d,equal]          & {H^1_{v}( B_1,\O^\times)}\arrow[d] \arrow[r]                        & {H^1_v(\wt B_1,\O^\times)}           \arrow[d]                \\
		0\arrow[r]&{\Hom_{\cts}(T_pB_1,K^\times)} \arrow[r] & {H^1_{v}(B_1\times \wt B_2,\O^\times)}  \arrow[r]& {H^1_{v}(\wt B_1\times \wt B_2,\O^\times)}.      
	\end{tikzcd}\]
	Specialisation at $0\in \wt B_2(K)$ defines a splitting of the vertical maps. Comparing to the \'etale cohomology, we see that the kernel of the sequence in the lemma gets identified with the kernel of 
	\[\Pic_{\et}(B_1)\to \Pic_{\et}(\wt B_1),\]
	which is $\wh B_1^\vee(K)$ by Theorem~\ref{t:Pic(wt B)} and Lemma~\ref{l:K-points-of-whB}.
\end{proof}
Combining Lemma~\ref{l:Pic(wtBxB')} with Theorem~\ref{t:Pic(wt B)} applied to $\wt B'^\vee\times \wt B$, this shows that pullback along $\wt B'^\vee\times \wt B\to B'^\vee\times\wt B$ defines a commutative diagram with left-exact middle row and column
\[
\begin{tikzcd}
&                                              & \Pic(\wt B) \arrow[r,equal]                        & \Pic(\overbar{B})\tf                                 \\
0 \arrow[r] & \wh B'(K) \arrow[r]                        & \Pic^\tau(B'^\vee\times\wt B) \arrow[r] \arrow[u] & \Pic(\overbar{B}'^\vee\times \overbar{B})\tf \arrow[u] \\
0\arrow[r]&\wh B'(K)\arrow[r]\arrow[u,equal]& {B'(\wt B).} \arrow[u]                     &                                    
\end{tikzcd}\]
Recall that by the classical theory of the Picard functor of abelian varieties, \[\Pic(\overbar{B}'^\vee\times\overbar{B})=\Pic(\overbar{B}'^\vee)\times \Pic(\overbar{B})\times \Hom(\overbar B,\overbar B').  \]
As we are working with $B'^\vee$-invariant line bundles, whose image in the first factor lands in $\overbar{B}'(k)$, we deduce that the image of $B'(\wt B)$ in the middle row generates a left-exact sequence
\[0\to \wh B'(K)\to B'(\wt B)\to \overbar{B}'(k)\tf\times \Hom(\overbar B,\overbar B')\tf.\]
Comparing to the short exact sequence of Lemma~\ref{l:K-points-of-whB}
\[0\to \wh B'(K)\to B'(K)\to \overbar{B}'(k)\tf\to 0,\]
this shows that $B'(K)$ and $\Hom(\overbar B,\overbar B')\tf$ generate all of $B'(\wt B)$, as we wanted to see.
\end{proof}

\begin{Remark}\label{r:strategy-via-p-divisible-groups}
	An alternative, and arguably more ``arithmetic'' strategy for the proof of $a)\Rightarrow c)$ is via $p$-divisible groups. This does not prove part 2 of the Theorem, but is independent of the results of \cite{heuer-diamantine-Picard}, although it still crucially invokes diamonds at some point.
	
	 We sketch the argument: Let $G$, $G'$ be the $p$-divisible group of $\mathfrak B$ and  $\mathfrak B'$. We can form the associated analytic universal covers $\wt{G}$ and  $\wt{G}'$ over $\Spa(K)$ in the sense of Scholze--Weinstein \cite[\S3.1]{ScholzeWeinstein}.
	Via untilting we can reduce to the case of $\Char K=0$.
	Suppose we have an isomorphism $\varphi:\wt B\to \wt B'$. 
	 By passing to topological torsion subgroups, $\varphi$ restricts to an isomorphism
	\[\varphi:\wt G\isomarrow \wt G'.\]
	Assume for a moment that $\varphi$ comes from an isogeny $G\to G'$. If this is the case, then $\varphi$ sends the Tate module of $G$ into the Tate module of $G'$. But these are precisely the Tate modules of $B$ and $B'$, respectively. Consequently, $\varphi$ induces a morphism of short exact sequences
		\[
		\begin{tikzcd}
			T_pB \arrow[d, "\varphi|_{T_pB}"'] \arrow[r] & \wt B \arrow[d, "\varphi"'] \arrow[r] & B \arrow[d, dashed] \\
			T_pB' \arrow[r]                     & \wt B' \arrow[r]                     & B'                 
		\end{tikzcd}\]
	of $v$-sheaves, 
		inducing the desired morphism $B\to B'$, which one easily checks is an isogeny.
	
	We are left to reduce to the case that $\varphi$ comes from an isogeny $G\to G'$. For this, we modify $B$ via Serre--Tate theory: One first proves that reduction defines an isomorphism
	\begin{equation}\label{eq:Hom-UC-of-p-div-groups}
	\Hom(\wt G_,\wt G')=\Hom(\wt G_{\O_K/\varpi},\wt G_{\O_K/\varpi}')=\Hom(G_{\O_K/\varpi},G_{\O_K/\varpi}')[\tfrac{1}{p}]
	\end{equation}
	(we thank Peter Scholze for explaining this to us).
	It follows that $\varphi$ induces a homomorphism
	\[\varphi:G_{\O_K/\varpi}\to  G'_{\O_K/\varpi}\]
	which possibly after enlarging the ideal $(\varpi)$ we can assume to be an isogeny. Let $H$ be its kernel $H$.
	Replacing $B$ by any lift $B_1$ of the abelian scheme $B_{\O_K/\varpi}/H$ to $\O_K$, the isogeny $ B_{\O_K/\varpi}\to  B_{\O_K/\varpi}/H$ lifts by $c)\Rightarrow a)$ to an isomorphism $\wt B\to \wt B_1$. We may thus without loss of generality assume $\varphi$ is an isomorphism.
	We now apply Serre--Tate lifting to
	\[ (B_{\O_K/\varpi},\varphi:G_{\O_K/\varpi}\isomarrow  G'_{\O_K/\varpi}).\]
	This produces a formal abelian scheme $B_2$ over $\O_K$ whose reduction is isomorphic to $B_{\O_K/\varpi}$ but whose $p$-divisible group is $G'$. In particular, again using $c)\Rightarrow a)$, we get an isomorphism
	\[\wt B_2\isomarrow\wt B \isomarrow\wt B'\]
	that restricts to the identity $\wt G_1\to \wt G_1$, as desired.\qed
\end{Remark}

Theorem~\ref{t:isomorphisms-between-ab-var-of-good-reduction} means that the perfectoid space $\wt B$ is determined by the $p$-isogeny class of the special fibre  $\overbar B$  of $B$. This raises the question whether one can explicitly recover the special fibre up to isogeny from $\wt B$ in a functorial way without the additional datum of $B$. One way to do so is by formal models, which requires one to define a natural analogue of a N\'eron model for $\wt B$. Another way is via the Picard functor of $\wt B$, which by Theorem~\ref{t:Pic-functor-wtB} is given by $\Pic_{\overbar B}^\diamond$. It thus suffices to note that passing from $\overbar B$ to $\overbar B^\diamond$ is fully faithful, by the following corollary of Theorem~\ref{t:isomorphisms-between-ab-var-of-good-reduction}:
\begin{Corollary}
	Let $B$ and $B'$ be two abeloids of good reduction over $K$. Then
	\[ H^0(\overbar{B}^\diamond\tf,\overbar{B}'^\diamond\tf)=H^0(\overbar{B},\overbar{B}')\tf.\] 
\end{Corollary}
\begin{proof}
	It is clear by functoriality that we have a map from right to left. To construct an inverse, let $\overbar{B}^\diamond\tf\to \overbar{B}'^\diamond\tf$ be any map. Recall from Proposition~\ref{p:ses-whB-B-barBtf} that there is a short exact sequence
	\[0\to \wh B\to B\to \overbar{B}^\diamond\tf\to 0.\]
	We form the composition
	\[ \wt B\to B\to \overbar{B}^\diamond\tf\to \overbar{B}'^\diamond\tf.\]
	It now suffices to see that this map lifts to $B'$, as by the universal lifting property Corollary~\ref{c:cor-to-top-univ-property} it then automatically lifts further to $\wt B'$, which implies the desired result by Theorem~\ref{t:isomorphisms-between-ab-var-of-good-reduction}. For this it suffices to see that $H^1_v(\wt B,\wh B')=1$. But this holds by Corollary~\ref{c:H^1(wt A,wh A')}.
\end{proof}
\subsection{Morphisms from universal covers of tori}
In order to get from the case of good reduction of Theorem~\ref{t:isomorphisms-between-ab-var-of-good-reduction} to the general case, our next goal is to describe morphisms between universal covers of Raynaud extensions. The first step in this direction is the following Proposition, which we prove in this subsection:
\begin{Proposition}\label{p:morphisms-between-univ-covers-of-Raynaud-extensions}
	Let $T\to E\to B$ and $T'\to E'\to B'$ be Raynaud extensions over $K$. Let \[\varphi:\wt E\to \wt E'\] be any morphism of perfectoid spaces over $K$ such that $\varphi(0)=0$. Then $\varphi$ is a homomorphism and there are unique homomorphisms $\varphi_T$ and $\varphi_B$ making the following diagram commute:
	\[
	\begin{tikzcd}
		0\arrow[r]&\wt T \arrow[d,"\varphi_T",dashed]\arrow[r] &  \wt E \arrow[r] \arrow[d,"\varphi"] & \wt B \arrow[d,"\varphi_B",dashed]\arrow[r]&0\\
		0\arrow[r]&\wt T'\arrow[r]&  \wt E' \arrow[r] &  \wt B'\arrow[r]&0.
	\end{tikzcd}
	\]
	The homomorphism $\varphi$ is uniquely determined by $(\varphi_T,\varphi_B)$.
\end{Proposition}
\begin{proof}
	We start by showing that if there are any morphisms of perfectoid spaces $\varphi_T$ and $\varphi_B$ making the diagram commute, then these determine $\varphi$ uniquely: Indeed, for any two $\varphi,\varphi'$ for which $\varphi_B$ agrees, the difference $\varphi-\varphi'$ factors through $\wt T'$. It thus suffices to see:
	\begin{Lemma}\label{l:fcts-of-wtE}
		The restriction map $H^0(\wt E,\wt T')\to H^0(\wt T,\wt T')$ is injective.
	\end{Lemma}
For the proof, we use a technical lemma on functions on tori that we again need later:
\begin{Lemma}[{\cite[Lemma~3.2]{heuer-Picard_perfectoid_covers}}]\label{l:wt-T}
	Let $T$ be any torus over $K$. Let $N:=\Hom(T,\G_m)$ be the character lattice . Let $U$ be any perfectoid space over $K$. Then
	\begin{enumerate}
		\item\label{e:O+-on-torus} $\O^+(\wt T\times U)=\O^+(U)$,
		\item\label{e:Ox-on-torus} $\O^{\times}(\wt T\times U)=\O^\times(U)\times \underline{N}\tf(U)$, 
		\item\label{e:bOx-on-torus}  $\bOx(\wt T\times U)=\bOx(U)\times \underline{N}\tf(U)$, where $\bOx=\O^\times/(1+\m\O^+)$ is from Definition~\ref{d:bOx}.
	\end{enumerate}
\end{Lemma}
\begin{proof}[Proof of Lemma~\ref{l:fcts-of-wtE}]
	Set  $M'^\vee:=\Hom(T',\G_m)$ and 
	consider the projection $\pi:\wt E\to \wt B$. As this is locally over affinoids $U\subseteq \wt B$ isomorphic to the projection map $\wt T\times U\to U$, it follows from Lemma~\ref{l:wt-T}.2 that we have a short exact sequence of sheaves on $\wt B$
	\[ 0\to \O_{\wt B}^\times\to \pi_{\ast}\O_{\wt E}^\times\to \underline{M'^\vee}\to 0.\]
	On global sections, this shows $\O^\times(\wt E)\subseteq \O^\times(\wt T)$. Since any morphism $\wt E\to \wt T'$ is determined by its compositions with the characters $\wt T'\to \G_m$, this implies the desired statement.
\end{proof}
	Next, we show that if there are morphism $\varphi_T$, $\varphi_B$ making the diagram commute, then $\varphi$ is a homomorphism: We have already seen that any morphism $\varphi_B:\wt B\to \wt B'$ such that $\varphi_B(0)=0$ is a homomorphism. To deduce that this is also true for $\varphi$, the uniqueness assertion applied to $E\times E$ shows that it suffices to prove that $\varphi_T$ is a homomorphism:
	\begin{Lemma}\label{l:H0(wtT,wtT')}
		Let $T$ and $T'$ be two tori over $K$. Then any morphism $\wt T\to \wt T'$ is the composition of a homomorphism and a translation. Moreover, we have $\Hom(\wt T,\wt T')=\Hom(T,T')\tf$.
	\end{Lemma}
\begin{proof}
	After a choice of splitting $T'=\G_m^r$, we may assume $T'=\G_m$. Then by Lemma~\ref{l:wt-T}.\ref{e:Ox-on-torus},
\[\textstyle H^0(\wt T,\wt \G_m)=\varprojlim_{[p]}H^0(\wt T,\G_m)=K^\times\times \Hom(T,\G_m)\tf.\]
The first factor accounts precisely for the translation by the points $\G_m(K)=K^\times$.
\end{proof}

For the proof of Proposition~\ref{p:morphisms-between-univ-covers-of-Raynaud-extensions}, it thus suffices to show that $\varphi:\wt E\to \wt E'$ induces morphisms $\varphi_T:\wt T\to \wt T'$ and $\varphi_B:\wt B\to \wt B'$ making the diagram commute. 
To produce $\varphi_T$, it suffices to show that any morphism $\wt T\to \wt B'$ is constant. Second, we need to see that any morphism
\[ \wt E\to \wt B'\]
factors through a map $\varphi_B:\wt B\to \wt B'$. Recall that locally over $U\subseteq \wt B$, we have $\wt E_{|U}=\wt T\times U$.
Both of these statements are therefore captured by the following more general result:
	\begin{Definition}
		An affinoid perfectoid space $X$ over $K$ is of good reduction if it can be written as an affinoid tilde-limit $U\sim\varprojlim_{i\in I} U_i$ of a cofiltered inverse system of generic fibres of smooth affine formal schemes over $\O_K$ such that $\varinjlim_{i\in I}\O(U_i)\to \O(U)$ has dense image.
	\end{Definition}
	\begin{Proposition}\label{p:Map(wtT,B)}
		Assume $\Char K=0$. Let $T$ be a torus, let $B$ be an abeloid of good reduction and let $U$ be a connected affinoid perfectoid space of good reduction. Then pullback defines an isomorphism
		\[H^0(\wt T\times U,B)=H^0(U,B).\]
	\end{Proposition}
\begin{proof}
	We evaluate  the  short exact sequence of Proposition~\ref{p:ses-whB-B-barBtf} at the projection $\wt T\times U\to U$ which is split by specialisation at $0\in \wt T$. This yields a morphism of left-exact sequences
	\[\begin{tikzcd}
		0 \arrow[r] & \wh B(U) \arrow[d] \arrow[r] & B(U) \arrow[d] \arrow[r] & \overbar{B}^\diamond\tf(U) \arrow[d] \arrow[r] & 0 \\
		0 \arrow[r] & \wh B(\wt T\times U) \arrow[r] & B(\wt T\times U) \arrow[r]& \overbar{B}^\diamond\tf(\wt T\times U) & 
	\end{tikzcd}\]
	As $U$ is affinoid perfectoid, the top sequence is short exact since $H^1_v(U,\wh B)=0$ by  \cite[Corollary~5.7]{heuer-Picard-good-reduction}. It therefore suffices to prove that the outer vertical maps are isomorphisms.
	
	For the map on the left, this is easy: We use the injection $\wh B\hookrightarrow \wh \G_m^{2d}$ of Proposition~\ref{p:analytic-Weil}, which in turn injects into a closed ball of radius $2d$. It now follows from Lemma~\ref{l:wt-T}.\ref{e:O+-on-torus} that any morphism $\varphi:\wt T\times U\to \wh B$ factors through the projection to $U$.
	
	We have thus reduced to seeing that the morphism on the right is an isomorphism. Since the definition of $\overbar{B}^\diamond\tf$ involves sheafification, it is not immediately clear how to check this directly. Instead, similarly to the proof of Theorem~\ref{t:isomorphisms-between-ab-var-of-good-reduction}, we again translate this into a statement about line bundles, but this time on the universal cover $\wt B$:
	
	By Theorem~\ref{t:Pic-functor-wtB}, the sheaf $\Pic_{\overbar{B}^\vee}^\diamond\tf$  represents the Picard functor of $\wt B^\vee$. It fits into an exact sequence
	\[ 0\to \overbar{B}^\diamond\tf\to \Pic_{\overbar{B}^\vee}^\diamond\tf\to \underline{\NS(\overbar{B}^\vee)}\tf\to 0. \]
	 Let $Q:=\underline{\NS(\overbar{B}^\vee)}\tf(U)$, then this means that we have a morphism of exact sequences
	\[\begin{tikzcd}
		0 \arrow[r] & \overbar{B}^\diamond\tf(U) \arrow[d] \arrow[r] & {H^1_{\et}(U\times \wt B^\vee,\bOx)} \arrow[d] \arrow[r] & {H^1_{\et}(U,\bOx)\times Q} \arrow[d] \arrow[r]&0\\
		0 \arrow[r] & \overbar{B}^\diamond\tf(\wt T\times U) \arrow[r] & {H^1_{\et}(\wt T\times U\times \wt B^\vee,\bOx)} \arrow[r] & {H^1_{\et}(\wt T\times U,\bOx)\times Q}\arrow[r]&0.
	\end{tikzcd}\]
	We now use the following result to show that the vertical morphisms are isomorphisms:

\begin{Proposition}[{\cite[Proposition~3.17]{heuer-Picard_perfectoid_covers}}]\label{p:bOx-on-wtTxV}
	Let $V$ be an affinoid perfectoid space of good reduction over $K$. Then pullback along $\wt T\times V\to V$ defines an isomorphism
	\[ H^1_{\et}(\wt T\times V,\bOx)=H^1_{\et}( V,\bOx).\]
\end{Proposition}
By our assumption on $U$, setting $V:=U$ shows immediately that the morphism on the right is an isomorphism. To see this for the morphism in the middle, recall that locally on $B^\vee$, the space $\wt B^\vee$ is made out of affinoid perfectoid spaces $V\subseteq \wt B^\vee$ of good reduction. More precisely, we can choose a cover of $\wt B^\vee$ by opens such that all intersections are of this form. We then obtain covers $\mathfrak U$ of $U\times \wt B^\vee$ and $\wt T\times \mathfrak U$ of $U\times \wt B^\vee$  with \cH-to-sheaf sequence
\[ 0\to \cH^1(\wt T\times \mathfrak U,\bOx)\to H^1(\wt T\times U\times \wt B,\bOx)\to   \cH^0(\wt T\times \mathfrak U,H^1(-,\bOx))\to \cH^2(\wt T\times\mathfrak U,\bOx).\]
Comparing to the cover $\mathfrak U$ via pullback and specialisation at $0\in \wt T$, we see from Proposition~\ref{p:bOx-on-wtTxV} that the third terms are isomorphic for both covers. It thus suffices to show
\[\cH^1(\wt T\times \mathfrak U,\bOx)=\cH^1(\mathfrak U,\bOx).\]
As $\bOx(\wt T\times U\times V)=\bOx(U\times V)\times M^\vee\tf(U\times V)$ by Lemma~\ref{l:wt-T}.3, it suffices to see that
\[ H^1_{\an}(\wt B^\vee\times V,\uZ)=0.\]
For this we first note that $\textstyle H^1_{\an}(\wt B^\vee,\uZ)=\varinjlim_{[p]}H^1_{\an}(B^\vee,\uZ)=0$
by \cite[Proposition 8.5]{BL-Degenerating-AV}, which we  recall below in
 Proposition~\ref{p:morph-from-T-to-B=0}. By the same limit argument, we have $H^1_{\an}(V,\uZ)=0$ by \cite[Lemma~8.11.c]{BL-stable-reduction-II}. The vanishing now follows from the Leray sequence for $\wt B^\vee\times V\to \wt B^\vee$.

This finishes the proof of Proposition~\ref{p:Map(wtT,B)},
\end{proof}
which in turn finishes the proof of Proposition~\ref{p:morphisms-between-univ-covers-of-Raynaud-extensions}.
\end{proof}
Before we go on, we  take a short detour to note that an independent and much simpler proof is possible if we are only interested in Main Theorem~\ref{t:intro-Hom(A,A')} rather than Theorem~\ref{t:morph-is-comp-intro}, that is if we know a priori that $\varphi$ is a homomorphism: 
In this case, for the existence of $\varphi_T$ and $\varphi_B$, it suffices to see that the induced homomorphism $\wt T\to \wt B'$ is trivial.
This is a perfectoid analogue of the following statement from the rigid theory:

Let $A$ be an abeloid variety with Raynaud uniformisation $A=E/M$ where $E$ is as before. Let $X_{\ast}(T)$ be the cocharacter lattice of $T$ and let $A^\vee$ be the dual abeloid to $A$.

\begin{Proposition}[{\cite[Proposition 8.5]{BL-Degenerating-AV}}]\label{p:morph-from-T-to-B=0}
	Let $A$ be an abeloid variety. Then
	\[ \Hom(\G_m,A)=\Hom(\G_m,T)=X_{\ast}(T)=H^1_{\an}(A^\vee,\uZ).\]
	In particular, we have $\Hom(\G_m,B)=H^1_{\an}(B^\vee,\uZ)=0$.
\end{Proposition}
We now have the following analogue for the perfectoid covers:
\begin{Proposition}\label{p:morphis-wtT-to-B}
	Let $A$ be an abeloid variety. Then
	\[\Hom(\wt \G_m,\wt A)=\Hom(\wt\G_m,\wt T)=X_{\ast}(T)[\tfrac{1}{p}]=H^1_{\an}(\wt A^\vee,\uZ).\]
	In particular, we have $\Hom(\wt \G_m,\wt B)=H^1_{\an}(\wt B^\vee,\uZ)=0$.
\end{Proposition}
\begin{proof}
	The second equality holds by Lemma~\ref{l:H0(wtT,wtT')}, the last by Proposition~\ref{p:morph-from-T-to-B=0} using that $H^1_{\an}(\wt A^\vee,\uZ\tf/\uZ)=0$.
	To see the first equality, we may by Theorem~\ref{t:tilting-abeloids} assume $\Char K=0$. The natural map
	\[ \Hom(\wt \G_m,E)=\Hom(\wt \G_m,A)\] 
	is an isomorphism as
	$\textstyle H^1_{\an}(\wt \G_m,\Z)=0$
	(see Proposition~\ref{p:ext-E-Z-vanishes} below). The exact sequence
	\[ 0\to  \Hom(\wt \G_m,T)\to \Hom(\wt \G_m,E)\to \Hom(\wt \G_m,B)\]
	now reduces us to showing that the last term vanishes.
	To see this, we pass to topological $p$-torsion parts to reduce to a statement about $p$-divisible groups: Consider  the diagram
	\begin{center}
		\begin{tikzcd}
			0\arrow[r]&\Hom(\G_m,B)\arrow[r]\arrow[d]&\Hom(\wt \G_m,B)\arrow[r]\arrow[d]& \Hom(\Z_p(1),B)\arrow[d,equal]\\
			0\arrow[r]&\Hom(\wh{\G}_m,\wh B)\arrow[r]& \Hom(\wh{\wt \G}_m,\wh B)\arrow[r]& \Hom(\Z_p(1),\wh B)
		\end{tikzcd}
	\end{center}
	of left-exact sequences. As the top left term vanishes by Proposition~\ref{p:morph-from-T-to-B=0}, it suffices to prove that the bottom left map has $p$-torsion cokernel.
	
	To see this, recall that the analytic Weil pairing of Proposition~\ref{p:analytic-Weil} sets up an injective homomorphism $\wh B\hookrightarrow \wh \G_m^{2d}$. The bottom row therefore injects into the sequence
	\[0\to {\Hom(\wh\G_m,\wh\G_m^{2d})}\to {\Hom(\wt{\wh \G}_m,\wh\G_m^{2d})}\to {\Hom(\Z_p(1),\wh{\G}^{2d}_m)}.\]
	But the left map has $p$-torsion cokernel since $\wt{\wh \G}_m=\wt{\mu_{p^\infty}}$ and by \eqref{eq:Hom-UC-of-p-div-groups}:
	\[\Hom(\wt{\mu_{p^\infty}},\wt{ \mu_{p^\infty}})=\Hom(\mu_{p^\infty,\O_K/p},\mu_{p^\infty,\O_K/p})\tf=\Hom(\Q_p/\Z_p,\Q_p/\Z_p)\tf=\Q_p,\]
	which is indeed equal to $\Hom(\wh \G_m,\wh\G_m)=\Z_p$ after inverting $p$.
\end{proof}
This concludes the alternative proof of  Proposition~\ref{p:morphisms-between-univ-covers-of-Raynaud-extensions} if $\varphi$ is a priori a homomorphism.

\subsection{Universal covers of Raynaud extensions in homological terms}
In light of Proposition~\ref{p:morphisms-between-univ-covers-of-Raynaud-extensions}, it remains to determine when a pair of homomorphisms $(\varphi_T,\varphi_B)$ consisting of $\varphi_T:\wt T\to \wt T'$ and $\varphi_B:\wt B\to \wt B'$ can be combined to a morphism $\varphi:\wt E\to \wt E'$.

To answer this question,  we start by reviewing some facts from \cite[\S3]{BL-Degenerating-AV} about the rigid analytic situation: Let $B$ be an abeloid variety of good reduction, $T$ a torus and let $M^\vee:=\Hom(T,\G_m)$ be the character lattice. For any $m\in M^\vee$, we get a natural map $\Ext^1_{\an}(B,T)\to \Ext^1_{\an}(B,\G_m)=B^\vee(K)$ (see  e.g.\ \cite[Theorem A.2.8]{Lutkebohmert_RigidCurves}).
This defines a natural map
\begin{equation}\label{eq:Raynaud-ext-=-morph-M^vee->B^vee}
\Ext^1_{\an}(B,T)\isomarrow \Hom(M^\vee,B^\vee(K)),
\end{equation}
which is in fact an isomorphism: Indeed, choose an isomorphism $T\cong\G_m^r$, then
\[\Ext^1_{\an}(B,T)\cong\Ext^1_{\an}(B,\G_m)^r=B^\vee(K)^r\cong\Hom(M^\vee,B^{\vee}(K)).\]
Thus any Raynaud extension $T\to E\to B$ gives rise to a homomorphism $f_E:M^\vee\to B^\vee(K)$ characterising its extension class.
Bosch--L\"utkebohmert \cite[Proposition~3.5]{BL-Degenerating-AV} use this to show:
\begin{Proposition}
 For any two Raynaud extensions $T\to E\to B$ and $T'\to E'\to B'$, 
\begin{equation}\label{eq:BL-Hom(E,E')}
\Hom(E,E')={\left\{
	\begin{array}{@{}l@{}l}
		\multicolumn{2}{@{}l}{\text{pairs }(\varphi_B,\varphi_{M^\vee}) \text{ where}}\\	
		\varphi_B&\in \Hom(B,B'),\\
		\varphi_{M^\vee}&\in \Hom(M'^\vee,M^\vee)
	\end{array}\middle|
	{\begin{tikzcd}[row sep =0.5cm]
			M^\vee  \arrow[r,"f_E"] & \overbar{B}^\vee(K)  \\
			M'^\vee \arrow[r,"f_{E'}"]\arrow[u,"\varphi_{M^\vee}"']& \overbar{B}'^\vee(K)  \arrow[u,"\varphi_B^\vee"']      
	\end{tikzcd}}
	\text{ commutes}
	\right \}}.
\end{equation}
\end{Proposition}
We would like to have an analogous description for $\Hom(\wt E,\wt E')$. For this we first deduce:
\begin{Proposition}\label{p:homological-interpretation-of-universal-cover}
	Let $B$ be an abelian variety of good reduction and $T$ a torus. Then we have an isomorphism of short exact sequences
	\[\begin{tikzcd}[column sep = 0.3cm]
		0 \arrow[r] &  \Ext^1_{\an}(B,\wh T) \arrow[d, "\sim"labelrotate]\arrow[r] & \Ext^1_{\an}(B,T) \arrow[d, "\sim"labelrotate] \arrow[r] &\Ext^1_{\an}(\wt B,\wt T) \arrow[d, "\sim"labelrotate] \arrow[r] & 0 \\
		0 \arrow[r] & \Hom(M^\vee,\wh{B}^\vee(K))  \arrow[r] & \Hom(M^\vee,B^\vee(K)) \arrow[r] & \Hom(M^\vee,B^\vee(k)[\tfrac{1}{p}])\arrow[r] & 0.
	\end{tikzcd}
	\]
	where the map on the top right sends an extension $[E]$ to its cover $[\wt E]$ from Theorem~\ref{t:perfectoid-covers-of-abelian-varieties}.1.
\end{Proposition}
\begin{proof}
	The middle isomorphism is \eqref{eq:Raynaud-ext-=-morph-M^vee->B^vee}.
	 The right isomorphism comes from Theorem~\ref{t:Pic(wt B)}  which shows
	\[\Ext^1_{\an}(\wt B,\wt T)=\Ext^1_{\an}(\wt B,T)=\Hom(M^\vee,B^\vee(k)[\tfrac{1}{p}])\]
	where the first equality follows from Lemma~\ref{l:ext-from-p-divisible-to-complete}.
	The bottom exact sequence is $\Hom(M^\vee,-)$ applied to the one of Lemma~\ref{l:K-points-of-whB}. 
	Exactness of the top row is shown in \cite[{Theorem~4.14}, Corollary~4.15]{heuer-Picard-good-reduction}.
	The last sentence holds because by Theorem~\ref{t:perfectoid-covers-of-abelian-varieties}.1, $\wt E$ can be described by forming the pullback of $T\to E\to B$ to $\wt B$ and then lifting uniquely to a $\wt T$-torsor.
\end{proof}
\begin{Corollary}\label{c:ext-classes-of-univ-covers-agree-iff-condition-on-reduction-of-M^vee->B^vee}
	For two extensions $[E], [E']\in \Ext^1(B,T)$, the following are equivalent:
	\begin{enumerate}
		\item The images $[\wt E],[\wt E']$ in $\Ext^1(\wt B,\wt T)$ define isomorphic extension classes.
		\item The associated maps $f,f':M^\vee\to B^\vee(K)$ agree up to $p$-torsion in the special fibre:
		\[\bar f_{E}=\bar f_{E'}:M^\vee\rightrightarrows B^{\vee}(K)\to \overbar B^\vee(k)[\tfrac{1}{p}].\] 
		\item There is an isogeny $\overbar E\to \overbar E'$ of $p$-power degree between the special fibres.
	\end{enumerate}
\end{Corollary}
\begin{proof}
	The equivalence of 1 and 2 follows from the diagram in Proposition~\ref{p:homological-interpretation-of-universal-cover}.
	
	The equivalence of 2 and 3 follows from $\Ext_{\mathrm{Zar}}^1(\overbar  B,\overbar T)=\Hom(M^\vee,B^\vee(k))$, which can be seen in exactly the same way as the isomorphism \eqref{eq:Raynaud-ext-=-morph-M^vee->B^vee} above.
\end{proof}

\subsection{Morphisms between universal covers of Raynaud extensions}

Putting together all the preparations of this section, we can now prove the following analogue of Bosch--L\"utkebohmert's description \eqref{eq:BL-Hom(E,E')} of $\Hom(E,E')$ at perfectoid level:

\begin{Theorem}\label{t:isoms-of-universal-covers-of-Raynaud-ext}
		Let $T\to E\to B$ and $T'\to E'\to B'$ be two Raynaud extensions over $K$ with corresponding morphisms $f_E:M^\vee\to B^\vee(K)$ and $f_{E'}:M'^\vee\to B'^\vee(K)$. Then
	\[\Hom(\wt E,\wt E')=\Hom(\overbar E,\overbar E')[\tfrac{1}{p}]\]
	\[={\left\{
		\begin{array}{@{}l@{}l}
		\multicolumn{2}{@{}l}{\text{pairs }(\varphi_B,\varphi_{M^\vee}) \text{ where}}\\	
		\varphi_B&\in \Hom(\overbar B,\overbar B')[\tfrac{1}{p}],\\
		\varphi_{M^\vee}&\in \Hom(M'^\vee,M^\vee)[\tfrac{1}{p}]
		\end{array}\middle|
		{\begin{tikzcd}
			M^\vee\tf  \arrow[r,"f_E"] & \overbar{B}^\vee(k)[\tfrac{1}{p}]  \\
			M'^\vee\tf \arrow[r,"f_{E'}"]\arrow[u,"\varphi_{M^\vee}"']& \overbar{B}'^\vee(k)[\tfrac{1}{p}]  \arrow[u,"\varphi_B^\vee"']      
			\end{tikzcd}}
		\text{ commutes}
		\right \}}.\]
	Moreover, any morphism $\wt E\to \wt E'$ over $K$ is a homomorphism composed with a translation.
\end{Theorem}
\begin{Corollary}
	The following are equivalent:
	\begin{enumerate}
		\item There is an isomorphism of perfectoid groups $\widetilde{E}\isomarrow \widetilde{E}'$ over $K$.
		\item There is an isomorphism of perfectoid spaces $\widetilde{E}\isomarrow \widetilde{E}'$ over $K$.
		\item There is an isogeny $\overbar{E}\to \overbar{E}'$  of $p$-power degree between the special fibres of $E, E'$.
		\item There is a pair $(\varphi_B,\varphi_{M^\vee})$ of an isogeny $\varphi_B:\overbar B\to \overbar B'$ of $p$-power degree and an isomorphism $\varphi_{M^\vee}:	M'^\vee\tf \to M^\vee\tf $  such that the above diagram commutes.
	\end{enumerate}
\end{Corollary}
\begin{proof}[Proof of Theorem~\ref{t:isoms-of-universal-covers-of-Raynaud-ext}]
	By Proposition~\ref{p:morphisms-between-univ-covers-of-Raynaud-extensions}, any homomorphism $\varphi:\wt E\to \wt E'$ induces a pair of homomorphisms 
	$\varphi_T:\wt T\to \wt T'$ and $\varphi_B:\wt B\to \wt B'$.
	By Theorem~\ref{t:isomorphisms-between-ab-var-of-good-reduction}, we have \[\Hom(\wt B,\wt B')=\Hom(\overbar B,\overbar B')[\tfrac{1}{p}].\]
	By Lemma~\ref{l:H0(wtT,wtT')}, we have
	\[\Hom(\wt T,\wt T')=\Hom(T,T')\tf=\Hom(M'^\vee,M^\vee)\tf.\]
	
	It remains to see that $\varphi_B$, $\varphi_T$ give rise to a  homomorphism $\wt E\to \wt E'$ if and only if the displayed diagram commutes.
	After pullback along $\varphi_B$ and pushout along $\varphi_T$, we can assume $T=T'$ and $B=B'$.
	 Then a morphism $\varphi$ that fits into the diagram
		\[
	\begin{tikzcd}
	0\arrow[r]&\wt T \arrow[d,equal]\arrow[r] &  \wt E \arrow[r] \arrow[d,"\varphi",dotted] & \wt B \arrow[d,equal]\arrow[r]&0\\
	0\arrow[r]&\wt T\arrow[r]&  \wt E' \arrow[r] &  \wt B\arrow[r]&0
	\end{tikzcd}
	\]
	is the same as an isomorphism of extensions of $v$-sheaves.
	By Lemma~\ref{l:interpretation-of-Ext1-as-extensions}, such an isomorphism exists if and only if the classes $[\wt E]$ and $[\wt E']$ in $\Ext^1(\wt B,\wt T)$ agree. By Corollary~\ref{c:ext-classes-of-univ-covers-agree-iff-condition-on-reduction-of-M^vee->B^vee}, this happens if and only if the reductions $\overbar f_E,\overbar f_{E'}:M^\vee\to B(K)\to B^\vee(k)\tf$ agree.
	
	In terms of extension classes in $\Ext^1(\wt B,\wt T)$, pullback along $\varphi_B$ corresponds to composition of $\overbar f_{E'}$ with $\varphi_B^\vee:B'^\vee(k)\tf\to B^\vee(k)\tf$, while pushout along $\varphi_T$ corresponds to precomposing $\overbar f_E$ with $\varphi_{M^\vee}:=\varphi_T^\vee:M'^\vee\to M^\vee$. This shows that $(\varphi_B,\varphi_{M^\vee})$ defines an isogeny $\wt E\to \wt E'$ if and only if the diagram displayed in the theorem commutes.

	The last sentence of the theorem has already been shown in Proposition~\ref{p:morphisms-between-univ-covers-of-Raynaud-extensions}.
\end{proof}
\begin{Corollary}\label{c:Hom(wtE,wtE')-fin-free}
	$\Hom(\wt E,\wt E')$ is a finite free $\Z[\tfrac{1}{p}]$-module. 
\end{Corollary}
\begin{proof}
	Both $\Hom(\overbar{B},\overbar{B}')$ and $\Hom(M'^\vee,M^\vee)$ are finite free $\Z$-modules.
\end{proof}

\section{Morphisms of universal covers: abeloids}
In this section, we prove the Main Theorems~\ref{t:intro-Hom(A,A')} and \ref{t:morph-is-comp-intro} from the introduction by describing all morphisms between universal covers of abeloids. 
\subsection{Lifting morphisms on perfectoid covers}
For any abeloid $A=E/M$, there is by Theorem~\ref{t:perfectoid-covers-of-abelian-varieties}.3 a short exact sequence of $v$-sheaves
\[ 0\to \wt E\to \wt A\to M_p/M\to 0.\]
As we have already described morphisms between universal covers of Raynaud extensions in the last section, the final step for describing morphisms of $\wt A$ is the following lifting result: 
\begin{Theorem}\label{t:isom-of-tilde-A-fixes-G}
	Let $A,A'$ be abeloid varieties over $K$ with associated Raynaud extensions $E,E'$.
	Then any morphism $\varphi:\widetilde{A}\to \widetilde{A}'$ such that $\varphi(0)=0$
	restricts to a morphism
	\begin{center}
		\begin{tikzcd}
			\wt E \arrow[d,hook] \arrow[r,dotted] & 	\wt E' \arrow[d,hook] \\
			\widetilde{A} \arrow[r,"\varphi"] & \widetilde{A}'
		\end{tikzcd}
	\end{center}
	between the covers of the Raynaud extensions.
	This restriction determines $\varphi$ uniquely. 
\end{Theorem}
\begin{Corollary}\label{c:isom-ex-iff-ex-isom-fixing-G_m,infty}
	Let $A$, $A'$ be abeloid varieties over $K$. If there is an isomorphism $\widetilde{A}\isomarrow \widetilde{A}'$, this restricts to an isomorphism
	$\wt{E}\isomarrow \wt{E}'$.
\end{Corollary}
\begin{proof}[Proof of Theorem~\ref{t:isom-of-tilde-A-fixes-G}]
	The proof will occupy the rest of this as well as the next subsection.
	
	We start by proving that the restriction of $\varphi$ to $\wt E$ determines $\varphi$ uniquely. which is easy:
	\begin{Lemma}\label{l:res-of-wtA-to-wtE}
		Let $Y$ be any perfectoid space. Then pullback defines an injection
		\[H^0(\wt A,Y)\to H^0(\wt E,Y). \]
	\end{Lemma}
	\begin{proof}
		We use the $M$-torsor $M_p\times \wt E\to \wt A$.
		Pulling back, the map $\wt A\to Y$ becomes an $M$-equivariant map
		$M_p\times \wt E\to Y$.
		The statement then follows from the following lemma:
	\end{proof}
	\begin{Lemma}\label{l:M-stable-morphisms-from-M_pxX}
		Let $Y$ and $Z$ be perfectoid spaces with $M$-action. Then  we have
		\[ H^0(M_p\times Z,Y)^{M}\subseteq H^0(Z,Y).\]
	\end{Lemma}
	\begin{proof}
		For any affinoid open $\Spa(R,R^+)\subseteq Z$, the restriction 
		$\Map_{\cts}(M_p,R)\to  \Map(M,R)$
		is injective, so $M_p\times Z\to Y$ is uniquely determined by its pullback $\phi:M\times Z\to M_p\times Z\to Y$.
		The $M$-equivariance means that it is already determined by its restriction to $Z$.
	\end{proof}
	
	It thus remains to prove the ``lifting'' assertion of Theorem~\ref{t:isom-of-tilde-A-fixes-G}. For this we apply $H^0(\wt E,-)$ to the short exact sequence of Theorem~\ref{t:perfectoid-covers-of-abelian-varieties}.3 for $A'$
	and get a long exact sequence
	\[0\to H^0(\wt E,\wt E') \to H^0(\wt E,\wt A')\to  H^0(\wt E,M'_p/M').\]
	The last term in turn sits in an exact sequence
	\[ H^0(\wt E,M')\to H^0(\wt E,M_p')\to H^0(\wt E,M_p'/M')\to  H^1_v(\wt E,M')\]
	of which the first two terms equal $M\to M_p$. The quotient $M_p/M\subseteq H^0(\wt E,M_p/M)$ accounts precisely for the different maps induced by translation on $\wt A$.
	It thus suffices to prove that the $H^1$-group on the right vanishes. This is the goal of the next subsection.
	
	\subsection{Triviality of $\mathbb Z$-torsors}\label{s:Z-torsors}
	As mentioned in the introduction, we can think of $E$ as the rigid analytic universal cover of $A$ in the following topological sense:
	\begin{Theorem}[Bosch--L\"utkebohmert, {\cite[Theorem~1.2.(c)]{BL-Degenerating-AV}}]\label{t:BL-12c} 	Let $0\to T\to E\to B\to 0$ be any Raynaud extension. Then
		$E$ is connected and we have $H^1_{\an}(E,\uZ)=0$.
	\end{Theorem}
	The goal of this subsection is to deduce an analogous result for the universal cover $\wt E$.

	\begin{Proposition}\label{p:ext-E-Z-vanishes}
		Let $0\to T\to E\to B\to 0$ be any Raynaud extension, then $H^1_v(\wt E,\underline{\Z})=0$.
	\end{Proposition}
	This follows easily from the following more general statement:
	\begin{Proposition}\label{p:Z-torsors-over-perfectoids}
		Let $X$ be any perfectoid space over $K$, then
		$H^1_{v}(X,\uZ)=H^1_{\an}(X,\uZ)$.
	\end{Proposition}
	\begin{proof}
		Consider for any $q\in K^\times$ with $|q|<1$ the short exact sequence of the Tate curve
		\[0\to q^{\Z}\to \G_m\to \G_m/q^{\Z}\to 0.\]
		Its long exact sequences for the analytic and $v$-topology fit into a commutative diagram
		\begin{center}
			\begin{tikzcd}[column sep = 0.2cm]
				{H^0_{v}(X,\G_m)} \arrow[r]           & {H^0_{v}(X,\G_m/q^\Z)} \arrow[r]             & {H^1_{v}(X,\underline{\Z})} \arrow[r] & {H^1_{v}(X,\G_m)} \arrow[r]           & {H^1_{v}(X,\G_m/q^{\Z})}           \\
				{H^0_{\an}(X,\G_m)} \arrow[r] \arrow[u] & {H^0_{\an}(X,\G_m/q^{\Z})} \arrow[u] \arrow[r] & {H^1_{\an}(X,\uZ)} \arrow[u] \arrow[r]   & {H^1_{\an}(X,\G_m)} \arrow[u] \arrow[r] & {H^1_{\an}(X,\G_m/q^{\Z})}. \arrow[u]
			\end{tikzcd}
		\end{center}
		The first two maps are isomorphisms. The fourth is an isomorphism by Theorem~\ref{t:KL-line-bundles}. The last map is injective by the Leray sequence, so the statement follows from the $5$-Lemma.
	\end{proof}
	\begin{proof}[Proof of Proposition~\ref{p:ext-E-Z-vanishes}]
		Combining  Proposition~\ref{p:Z-torsors-over-perfectoids} and Theorem~\ref{t:BL-12c}, we see that we have \[\textstyle H^1_{v}(\wt E,\underline{\Z})=H^1_{\an}( \wt E,\underline{\Z})=\varprojlim_{[p]} H^1_{\an}(E,\underline{\Z})=0,\]
		as we wanted to see.
	\end{proof}

This finishes the proof of Theorem~\ref{t:isom-of-tilde-A-fixes-G}.
\end{proof}

\begin{Corollary}\label{c:ext(E,Z)}
	We have $\Ext^1_v( \wt E,\underline{\Z})=\Ext^1_v( E,\underline{\Z})=0$.
\end{Corollary}
\begin{proof}
	Since $\wt E$ is connected and $\underline{\Z}$ is a  constant sheaf, Lemma~\ref{l:E_2-of-discrete-torsor-over-connected} says that we have an injection $\Ext^1_{v}( \wt E,\underline{\Z})\hookrightarrow H^1_{v}( \wt E,\underline{\Z})$. The same works for $E$.
\end{proof}

\begin{Remark}
	The $\Ext^2_v$ groups do not vanish: Indeed, it follows from the proposition and Lemma~\ref{l:isoms-of-Zp/Z}.\ref{enum:ext 1(uzp,z)} that the $p$-divisible group $E[p^\infty]=\Ext^1_v(T_pE,\uZ)$ injects into $\Ext^2_v(E,\underline{\Z})$.
\end{Remark}
\begin{Remark}
We consider Theorem~\ref{t:isom-of-tilde-A-fixes-G} to be an analogue of the rigid statement that any morphism $A\to A'$ lifts to a morphism $E\to E'$. This analogy is perhaps more transparent in the following variant in terms of $X:=M_p\times \wt E$, which was mentioned in the introduction:
\begin{Corollary}
	For $n\gg 0$, the map $[p^n]\varphi$ lifts uniquely to a homomorphism
	\[ X\to X'.\]
\end{Corollary}
\begin{proof}
	This follows via long exact sequences and using Proposition~\ref{p:higher-cohomolgy-of-profinite=0}.2 which shows that $\Ext^1_v(M_p,M')=\Hom(M,M')\otimes \Ext^1_v(\uZ_p,\uZ)=\Hom(M,M')\otimes\Q_p/\Z_p$ is $p$-torsion.
\end{proof}
	In particular, $M_p\times \wt E\to \wt A$
	is the universal cover of $\wt A$ for $\Z$-coefficients. This is the sense in which $X\to \wt A$ is an analogue of the cover $E\to A$ in the rigid analytic setting.
\end{Remark}

\subsection{Main Theorem}
Using Theorem~\ref{t:isoms-of-universal-covers-of-Raynaud-ext}, we can now prove Main Theorems~\ref{t:intro-Hom(A,A')} and \ref{t:morph-is-comp-intro} from the introduction:
\begin{Theorem}\label{t:main-theorem}
	Let $A=E/M$ and $A'=E'/M'$ be two abeloids over $K$.
		Then  any morphism $\wt A\to \wt A'$ over $K$ is the composition of a homomorphism and a translation. Moreover, let  $g:M\to \wt E$ and $g':M\to \wt E'$ be any choices of lifts, then there is a natural isomorphism
	\begin{equation}\label{eq:isog-in-main-thm}
	\Hom(\wt A,\wt A')={\left\{\begin{array}{@{}l@{}l}	
		\varphi_E&\in \Hom(\wt E,\wt E')\end{array}\middle|\, 
		\begin{array}{lr}
		\exists \,\varphi_M:M[\tfrac{1}{p}]\to M'[\tfrac{1}{p}]\\	
		\text{such that the following }\\
		\text{diagram commutes up}\\\text{to topological $p$-torsion:}\end{array}
		\quad 
		{\begin{tikzcd}
			M[\tfrac{1}{p}] \arrow[d,"\varphi_M"'{yshift=0.36ex}] \arrow[r,"g"] & \wt E \arrow[d,"\varphi_E"'{yshift=-0.36ex}] \\
			M'[\tfrac{1}{p}] \arrow[r,"g'"]& \wt E'   
			\end{tikzcd}}
		\right \}}
\end{equation}
	(see Definition~\ref{d:G-hat} for ``topological $p$-torsion''). 
	In particular, the following are equivalent:
	\begin{enumerate}
		\item There is an isomorphism of perfectoid groups $\varphi:\wt A\isomarrow \wt A'$ over $K$.
		\item There is an isomorphism of perfectoid spaces $\varphi:\wt A\isomarrow \wt A'$ over $K$.
		\item There is a triple $(\varphi_B,\varphi_T,\varphi_M)$ consisting of
		\begin{enumerate}
			\item an isogeny $\varphi_B:\overbar B\to \overbar B'$ of $p$-power degree between the special fibres of $B$, $B'$,
			\item an isogeny $\varphi_T:T\to T'$ of $p$-power degree s.t.\ the following diagram commutes:
			\[
			\begin{tikzcd}
			M^\vee  \arrow[r,"\overbar f"]& \overbar{B}^\vee(k)[\tfrac{1}{p}]  \\
			M'^\vee \arrow[r,"\overbar f'"]\arrow[u,"\varphi_{M^\vee}"',"\sim"labelrotate]           & \overbar{B}'^\vee(k)[\tfrac{1}{p}]\arrow[u,"\varphi_B^\vee"',"\sim"labelrotate]
			\end{tikzcd}
			\]
			where $\varphi_{M^\vee}$ is the morphism induced on the character lattices by $\varphi_T:T\to T'$. In particular, by Theorem~\ref{t:isoms-of-universal-covers-of-Raynaud-ext}, $(\varphi_B,\varphi_{M^\vee})$ induces an isogeny $\varphi_E:\wt E\to \wt E'$,
			\item an isogeny of $p$-power degree $\varphi_M:M\rightarrow M'$ for which 
			\[\varphi_E\circ g- g'\circ\varphi_M:M\to \wt E(K)\to E(K)\]
			has image contained in the open subgroup $\wh{E}(K)$.
		\end{enumerate}
	\end{enumerate}
\end{Theorem}

\begin{Corollary}\label{c:Hom-wtA-wtA'-fg}
	$\Hom(\wt A,\wt A')$ is a finite free $\Z[\tfrac{1}{p}]$-module.
\end{Corollary}
\begin{proof}
	The $\Z\tf$-module $\Hom(\wt E,\wt E')$ is finite free by Corollary~\ref{c:Hom(wtE,wtE')-fin-free}.
\end{proof}
\begin{Remark}
	We will see in \cite{heuer-Picard_perfectoid_covers} that the N\'eron--Severi group of $\wt A$ is precisely the symmetric subgroup of $\Hom(\wt A,\wt A^\vee)$. In particular, it is finitely generated by Corollary~\ref{c:Hom-wtA-wtA'-fg}.
\end{Remark}
\begin{Remark}
In the case that $E=E'$, condition $(c)$ can be rephrased without the choice of lifts $g$ and $g'$, by asking that $m-\varphi_M(m)\in \widehat{E}(K)\tf$ for all $m \in M$. Since we can picture $\widehat{E}(K)$ as being a small open neighbourhood of the identity in $E$, this roughly says  that the two lattices $M$ and $M'$ are ``$p$-adically close'' to each other in the ambient space $E$.
\end{Remark}

\begin{Example}\label{ex:non-trivial-morph-between-A_1-and-A_2}
	\begin{enumerate}
		\item 
		In order to illustrate condition $(c)$, we note that for Tate curves, it means the following: For any $q,q'\in K^\times$ with $|q|,|q'|<1$, there is an isomorphism 
		\[\widetilde{\G_m/q^\Z}\isomarrow \widetilde{\G_m/q'^\Z}\]
		of perfectoid spaces over $K$ if and only if $q'\in (1+\m)\cdot q^{\Z[\frac{1}{p}]}$.
		\item The isomorphisms between universal covers are really a feature of the $p$-adic universal covers; the analogous result for intermediate perfectoid covers does not hold. For example, for Tate curves in characteristic $p$, there is an isomorphism
		\[(\G_m/q^\Z)^\perf\isomarrow (\G_m/q'^\Z)^\perf\]
		if and only if $q'\in q^{\Z[\frac{1}{p}]}$. Indeed, this can be seen using the short exact sequence
		\[0\to q^{\Z}\to \wt\G_m\to (\G_m/q^{\Z})^\perf\to 0,\]
		the fact that $\Ext^1_v(\wt\G_m,\Z)=0$ by Corollary~\ref{c:ext(E,Z)}, and $\End(\wt\G_m)=\Z\tf$ by Lemma~\ref{l:H0(wtT,wtT')}.
	\item 
	In order to give a more elementary perspective on the additional isomorphisms between universal covers, we now describe them on the level of the underlying topological group in the case of Tate curves:
	Let $q,q'\in K^\times$ be such that $|q|=|q'|<1$ and such that there is $z\in 1+\m$ with $q=q'z$. We construct a continuous homomorphism
	\[ \textstyle\varprojlim_{[p]}K^\times/q^\mathbb Z\to K^\times/q'^{\mathbb Z}\]
	as follows: Let $(x_n)_{n\in\N}$ be an element of the left hand side. Lifting each element individually to $K^\times$ yields a sequence $(\tilde{x}_n)_{n\in\N}$ in $K^\times$ satisfying
	\[\tilde{x}_{n+1}^p\tilde{x}^{-1}_n\in q^\Z\quad \text{ for all }n\in \N.\]
	We claim that $(\tilde{x}_n^{p^n})_{n\in\N}$ is a convergent sequence in the topological group $K^\times/q'^{\mathbb Z}$, with limit independent of the choice of lift. Indeed, if $\tilde{x}'_n$ is any other lift of $x_n$, then $\tilde{x}'_n = \tilde{x}_n\cdot q^k$  for some $k\in \Z$.
	Consequently, we have
	\[\tilde{x}'^{p^n}_n=\tilde{x}^{p^n}_n z^{kp^n}q'^{kp^n}.\]
	Applying this to $\tilde{x}_n':=\tilde{x}_{n+m}^{p^m}$ for any $m>0$, we see that inside $K^\times/q'$ we have
	\[\tilde{x}^{p^{n+m}}_{n+m}/\tilde{x}^{p^n}_n=z^{kp^n}\to 1 \quad \text{ for }n\to \infty.\]
 Hence $(\tilde{x}_n^{p^n})_{n\in\N}$ is a Cauchy sequence in $K^\times/q'^\Z$. Sending  $(x_n)_{n\in\N}$ to the limit defines a map
	\[ f:\textstyle\varprojlim_{[p]}K^\times/q^\Z\to K^\times/q'^\Z.\]
	Taking the inverse limit over $[p]$ on the target  gives the desired isomorphism. 
	
	That said, it is of course not clear from this computation that $f$ comes from an isomorphism of perfectoid groups, and there are indeed many automorphisms of the topological  group $\textstyle\varprojlim_{[p]}K^\times/q^\Z$ that do not come from morphisms of perfectoid groups.
		\item We get even more isomorphisms if we leave the category of perfectoid spaces over $K$, and instead consider absolute isomorphisms of diamonds over $\F_p$. This is a general phenomenon that is already true for the abeloids themselves: For example, any Tate curve over  $K$ of characteristic $p$ is already defined over a subfield generated by $q$ isomorphic to $K_0:=\F_p(\!(q^{1/p^\infty})\!)$, and the completed algebraic closure of $K_0$ has many field automorphisms over $\F_p$, identifying Tate curves for many different $q$.
		
		However, it is clear that even in this absolute setting there are abeloids of the same dimension that cannot have isomorphic universal covers, as is evidenced by invariants of $\wt A$ like, say, the Picard rank (see Theorem~\ref{t:Pic(wt B)}).
	\end{enumerate}
\end{Example}

\begin{proof}[Proof of Theorem~\ref{t:main-theorem}]
	That any morphism $\varphi:\wt A\to \wt A'$ such that $\varphi(0)=0$ is a homomorphism follows from the corresponding statement in Theorem~\ref{t:isoms-of-universal-covers-of-Raynaud-ext} by applying Theorem~\ref{t:isom-of-tilde-A-fixes-G} to the map $m\circ (\varphi\times \varphi)-\varphi\circ m:\wt A\times \wt A\to \wt A'$ which is trivial if and only if $\varphi$ is a homomorphism.
	
	To show the description of $\Hom(\wt A,\wt A')$,  we start with a homomorphism $\varphi:\wt A\to \wt A'$ of perfectoid groups.
	By Theorem~\ref{t:isom-of-tilde-A-fixes-G}, $\varphi$ restricts to a homomorphism  $\varphi_E:\wt E\to \wt E'$. Using the sequence of Theorem~\ref{t:perfectoid-covers-of-abelian-varieties}.3, we can thus interpret $\varphi$ as a morphism of short exact sequences
		\begin{center}
		\begin{tikzcd}
		0\arrow[r]&\wt E\arrow[r] \arrow[d,"\varphi_E",dashed] &  \widetilde{A} \arrow[r] \arrow[d,"\varphi"] &  M_p/M \arrow[d,"\varphi_{M}",dashed]\arrow[r]&0\\
		0\arrow[r]&\wt E'\arrow[r]&  \widetilde{A}' \arrow[r] &  M'_p/M'\arrow[r]&0.
		\end{tikzcd}
	\end{center}
	To describe $\varphi_{M}$, we first note that
	Lemma~\ref{l:isoms-of-Zp/Z}.\ref{enum:End(uZp/uZ)} implies that we have 
	\[\Hom(M_p/M,M'_p/M')=\Hom(M,M')[\tfrac{1}{p}],\]
	which shows that $\varphi_{M}$ is induced by an isogeny $\varphi_M:M\tf\to M'\tf$. 
	
	We note in passing that this gives another way to prove the uniqueness assertion of Theorem~\ref{t:isom-of-tilde-A-fixes-G} in the case that $\varphi$ is a homomorphism: By Lemma~\ref{l:H^0(Z_p/Z,G)}, every homomorphism $M_p/M\to \wt E'$ is trivial. Second, that $\varphi_M$ is determined by $\varphi_E$ as asserted by Theorem~\ref{t:isom-of-tilde-A-fixes-G} is clear from the short exact sequence
	\[ 0\to M\to M_p\times \wt E\to \wt A\to 0.\]
	
	It remains to explain why the commutativity condition in \eqref{eq:isog-in-main-thm} holds: In terms of extension classes in $\Ext^1_v(M_p/M,\wt E')$, given $\varphi_E$ and $\varphi_{M}$, for these to extend to a morphism $\varphi$ in the above diagram, it is a necessary and sufficient condition to have
	\[\varphi_{E\ast}[\wt A]=\varphi_{M}^{\ast}[\wt A'].\]
	More precisely, this means that in the following commutative diagram
	
	\[
	\begin{tikzcd}
	0\arrow[r]&\Hom(M_p,\wt E) \arrow[r] \arrow[d,"\varphi_{E\ast}"]&  \Hom(M,\wt E) \arrow[r]\arrow[d,"\varphi_{E\ast}"]& \Ext^1(M_p/M,\wt E)\arrow[d,"\varphi_{E\ast}"]\\
	0\arrow[r]&\Hom(M_p,\wt E') \arrow[r] &  \Hom(M,\wt E') \arrow[r]& \Ext^1(M_p/M,\wt E')\\
	0\arrow[r]&\Hom(M'_p,\wt E')\arrow[r] \arrow[u,"\varphi_{M}^{\ast}"']&  \Hom(M',\wt E')\arrow[u,"\varphi_{M}^{\ast}"'] \arrow[r] & \Ext^1(M'_{p}/M',\wt E')\arrow[u,"\varphi_{M}^{\ast}"'],
	\end{tikzcd}\]
	the terms $[\wt A]$ in the top right and $[\wt A']$ in the bottom right are sent to the same element in the middle. By Corollary~\ref{c:interpret-of-boundary-map-for-wt-E} and exactness of the rows, this is equivalent to the difference
	\begin{equation}\label{eq:commutativity-condition-in-proof-of-Main-Theorem}
	\varphi_M^{\ast}g'-\varphi_{E\ast}g=g'\circ \varphi_M-\varphi_E\circ g\in \Hom(M_p,\wt E')
	\end{equation}
	landing in the left of the middle row. Since
	\[\Hom(M_p,\wt E')=\Hom(M,\Hom(\uZp,\wt E'))=\Hom(M,\wh{\wt E'})\]
	by Proposition~\ref{p:morphisms-from-Z,Z_p,Z_p/Z} and Definition~\ref{d:G-hat},
 	this is precisely the condition on commutativity up to topological $p$-torsion difference displayed in \eqref{eq:isog-in-main-thm}.
	
	Finally, if $\varphi$ is an isomorphism, then clearly $\varphi_E$ is an isomorphism and $\varphi_M$ is an isogeny of degree $p$. Moreover, condition~\eqref{eq:commutativity-condition-in-proof-of-Main-Theorem} is equivalent to condition $3(c)$. This together with Theorem~\ref{t:isoms-of-universal-covers-of-Raynaud-ext} shows that conditions 1 and 3 in the Theorem are equivalent.
\end{proof}

\begin{Remark}
	The left-exact sequence in the proof is in fact a short  exact sequence
	\[0\to \Hom(M_p,\wt T)\to \Hom(M,\wt T)\to \Ext^1_v(M_p/M,\wt T)\to \Ext^1_v(M_p,\wt T)=0.\] This does not mean, however, that every extension class appears as $[\wt A]$ for some $A$, as not every $M\to T$ corresponds to an abeloid. For instance, $\wt E\to  \widetilde{A} \to M_p/M$ is never split.
\end{Remark}

\subsection{Adelic covers}
Finally, in this section, we discuss analogues of Theorem~\ref{t:main-theorem} for the adelic cover
\[ \wtt A=\textstyle\varprojlim_{[N]}A,\]
where $N$ ranges through $\N$. 
We first prove the following analogue of Theorem~\ref{t:isoms-of-universal-covers-of-Raynaud-ext}:
\begin{Theorem}\label{t:Hom(wttE,wttE')}
	Let $E$ and $E'$ be Raynaud extensions over $K$ with special fibres $\overbar E$ and $\overbar{E}'$ over $k$. Then
	\[\Hom(\wtt E,\wtt E')=\Hom(\overbar E,\overbar E')\otimes \Q.\]
\end{Theorem}
\begin{proof}
	This is immediate from from Theorem~\ref{t:isoms-of-universal-covers-of-Raynaud-ext} and the following lemma:
\end{proof}
\begin{Lemma}\label{l:Hom(wtt A)-vs-Hom(wtA)}
	Let $E$, $E'$ be Raynaud extensions over $K$. Then we have
	\[\Hom(\wtt E,\wtt E')=\Hom(\wt E,\wt E')\otimes_{\Z}\Q.\]
	The analogous result is true for abeloids $A$ and $A'$, i.e.\ we have
	\[\Hom(\wtt A,\wtt A')=\Hom(\wt A,\wt A')\otimes_{\Z}\Q.\]
\end{Lemma}
\begin{Remark}
This is in contrast to the situation for the cover $\wt A\to A$, where we have seen that the map
$\Hom(A,A')\tf\to \Hom(\wt A,\wt A')$
is still injective, but  often not surjective. As the proof shows, the reason is that $\wt A\to A$ is pro-$p$, whereas $\wtt A\to \wt A$ is pro-coprime-to-$p$.
\end{Remark}

\begin{Definition}\label{d:coprime-to-p-Tate-module}
	For any abeloid $A$, we denote by $T^{p}A:=\varinjlim_{(N,p)=1}A[N]=TA/T_pA$ the adelic Tate module away from $p$. This fits into a short exact sequence of $v$-sheaves
	\begin{equation}\label{eq:ses-of-wttA}
		0\to T^pA\to \wtt A\to \wt A\to 0.
	\end{equation}
	For a Raynaud extension $E$, we similarly define $T^{p}E$ which fits into an analogous sequence.
\end{Definition}
\begin{proof}
	Taking the limit over $[N]$ for $N\in \N$ with $(N,p)=1$ defines a natural morphism
	\[\Hom(\wt E,\wt E')\otimes \Q \to \Hom(\wtt E,\wtt E').\]
	To see that this is an isomorphism, it suffices to see that for any homomorphism $\varphi:\wtt E\to \wtt E'$ there is $N\in \N$ such that $[N]\varphi$ induces a morphisms of the short exact sequences from \eqref{eq:ses-of-wttA}:
	\[
		\begin{tikzcd}
			0\arrow[r]&T^{p}E\arrow[r] \arrow[d,dashed] &  \wtt E \arrow[r] \arrow[d] &  \wt E \arrow[d,dashed]\arrow[r]&0\\
			0\arrow[r]&T^{p}E'\arrow[r]&  \wtt E' \arrow[r] &  \wt E'\arrow[r]&0.
		\end{tikzcd}\]
	To see this, we need to prove that any homomorphism $T^pE\to \wt E'$ is torsion. But this follows from the fact that $p$ is an isomorphism on $T^pE$ and thus $\varprojlim_{[p]}$ induces an isomorphism
	\[\Hom(T^pE,\wt E')=\Hom(T^pE,E'),\]
	which indeed is torsion by Lemma~\ref{l:morphism-from-pro-coprime-to-p}.
	The case of abeloids is completely analogous.
\end{proof}

To deduce the adelic version of the Main Theorem, we first need to introduce the adelic analogue of the topological $p$-torsion subgroup $\wh G\subseteq G$ of an adic group $G$ from  Definition~\ref{d:G-hat}:
\begin{Definition}\label{d:top-tor}
	Let $G$ be any rigid or perfectoid group over $K$. We denote by $G^\tt\subseteq G$ the topological torsion subsheaf, defined as the image in $v$-sheaves of the evaluation map
	\[ \HOM(\underline{\whZ},G)\to G, \quad \psi\mapsto \psi(1).\]
	Let $F$ be another rigid or perfectoid group. We say that two maps $f,g:F\to G$ agree up to topological torsion if $f-g$ factors through $G^\tt$.
	In particular, this gives a notion of what it means for a diagram of abelian $v$-sheaves to commute up to topological torsion.
\end{Definition}
We always have  $\wh G\subseteq G^\tt$. In all situations we encounter, the cokernel will be torsion:
\begin{Proposition}[{\cite[Proposition~2.19.2]{heuer-diamantine-Picard}}]\label{p:-Att-for-abeloids}
	Let $A$ be an abeloid over $K$. Then 
	\[
	A^\tt=\wh A\cdot \bigcup_{\substack{N\in \N\\(N,p)=1}} A[N].
	\]
\end{Proposition}

\begin{Theorem}\label{t:main-theorem-adelic-case}
	Let $A=E/M$ and $A'=E'/M'$ be two abeloids over $K$ and let  $g:M\to \wtt E$ and $g':M\to \wtt E'$ be any choices of lifts, then there is a natural isomorphism
	\begin{equation*}
		\Hom(\wtt A,\wtt A')={\left\{\begin{array}{@{}l@{}l}	
				\varphi_E&\in \Hom(\wtt E,\wtt E')\otimes \Q\end{array}\middle|\, 
			\begin{array}{lr}
				\exists \,\varphi_M:M\otimes \Q\to M'\otimes \Q\\	
				\text{such that the following }\\
				\text{diagram commutes up}\\\text{to topological torsion:}\end{array}
			\quad 
			{\begin{tikzcd}
					M\otimes \Q \arrow[d,"\varphi_M"'{yshift=0.36ex}] \arrow[r,"g"] & \wtt E \arrow[d," \varphi_E"'{yshift=-0.36ex}] \\
					M'\otimes \Q \arrow[r,"g'"]& \wtt E'   
			\end{tikzcd}}
			\right \}}
	\end{equation*}
(see  Definition~\ref{d:top-tor} for ``topological torsion'').
	In particular, the following are equivalent:
	\begin{enumerate}
		\item There is an isomorphisms of perfectoid groups $\varphi:\wtt A\isomarrow \wtt A'$ over $K$.
		\item There is a triple $(\varphi_B,\varphi_T,\varphi_M)$ consisting of
		\begin{enumerate}
			\item a quasi-isogeny $\varphi_B:\overbar B\to \overbar B'$ between the special fibres over the residue field $k$.
			\item a quasi-isogeny $\varphi_T:T\to T'$  s.t.\ the following diagram commutes:
			\[
			\begin{tikzcd}
			M^\vee \otimes \Q \arrow[r,"\overbar f"]& \overbar{B}^\vee(k)\otimes \Q   \\
			M'^\vee \otimes \Q \arrow[r,"\overbar f'"]\arrow[u,"\varphi_{M^\vee}"',"\sim"labelrotate]           & \overbar{B}'^\vee(k)\otimes \Q \arrow[u,"\varphi_B^\vee"',"\sim"labelrotate],
			\end{tikzcd}
			\]
			where $\varphi_{M^\vee}$ is the map induced by $\varphi_T$ on character lattices.
			Equivalently, $(\varphi_B,\varphi_T)$ defines a quasi-isogeny $\overbar E\to \overbar E'$ which in turn gives an isomorphism $\varphi_E:\wtt E\to \wtt E'$.
			\item an isomorphism $\varphi_M:M\otimes \Q\rightarrow M'\otimes \Q$ for which the difference 
			\[\varphi_E\circ g- g'\circ\varphi_M:M\otimes \Q\to \wtt E'(K)\to E'(K)\]
			has image contained in $E'^\tt(K)$, i.e.\ is topologically torsion.
		\end{enumerate}
	\end{enumerate}
\end{Theorem}
\begin{proof}[Proof of Theorem~\ref{t:main-theorem-adelic-case}]
	By Theorem~\ref{t:main-theorem}, sending a pair $(\varphi:\wt E\to \wt E',\varphi:M\tf\to M'\tf)$ to the difference $\varphi_E\circ g-g'\circ\varphi_M:M\tf\to E'(K)\otimes\tf$ defines a Cartesian square
	\[
	\begin{tikzcd}
		{\Hom(\wt E,\wt E')\times \Hom(M\tf,M'\tf)} \arrow[r]           & {\Hom(M\tf,E'(K)\tf)}               \\
		{\Hom(\wt A,\wt A')} \arrow[r] \arrow[u] & {\Hom(M\tf,\wh E'(K)\tf)}. \arrow[u]
	\end{tikzcd}\]
	After applying the exact functor $-\otimes_{\Z}\Q$, the diagram is still Cartesian. This characterises $\Hom(\wt A_1,\wt A_2)\otimes \Q$, which by Lemma~\ref{l:Hom(wtt A)-vs-Hom(wtA)} equals $\Hom(\wtt A_1,\wtt A_2)$. 
	
	The same argument applies to $\Hom(\wt E_1,\wt E_2)\otimes \Q=\Hom(\wtt E_1,\wtt E_2)$ and the description of $\Hom(\wt E_1,\wt E_2)$ given in Theorem~\ref{t:isoms-of-universal-covers-of-Raynaud-ext}. This shows that 1 and 2 are equivalent.
\end{proof}

It remains to prove Theorem~\ref{t:morph-is-comp-intro} in the adelic case, which is less straight-forward:
\begin{Theorem}\label{t:morphis-is-comp-of-hom-with-trans-adelic-case}
	Let $A=E/M$ and $A'=E'/M'$ be two abeloids over $K$. Then any morphism $\wtt A\to \wtt A'$ over $K$ is the composition of a homomorphism and a translation.
\end{Theorem}

 It is possible to prove this in exactly the same way as in Theorem~\ref{t:main-theorem}, replacing each $p$-adic cover by an adelic cover, each $\tf$ by $\otimes \Q$, etc. However, in order to avoid having to do everything twice, we instead explain in this section how one can deduce everything from the $p$-adic case, which we think is conceptually interesting in itself.

We first prove a version of Lemma~\ref{l:Hom(wtt A)-vs-Hom(wtA)} for morphisms that are not necessarily linear, which is more work: For this we use results on the Picard functor of $\wtt A_2$ from \cite[\S4]{heuer-diamantine-Picard}.
\begin{Lemma}\label{l:fact-of-morph-from-wttA}
	Let $A_1$, $A_2$ be abeloids over $K$, then the following natural map is an isomorphism:
	\[ \textstyle H^0(\wtt A_1,A_2)\isomarrow \varinjlim_{[N]}H^0(\wt A_1,A_2),\]
	where the system on the right ranges through $[N]$ on $\wt A_1$ for $N\in \N$ with $(N,p)=1$.
\end{Lemma}
\begin{proof}
	Let $\varphi:\wtt A_1\to A_2$ be any morphism. We need to see that this factors through some projection to $\wt A_1$. Let $A_2^\tt\subseteq A_2$ be the topological torsion subgroup and consider the short exact sequence of $v$-sheaves
	\begin{equation}\label{eq:A_2^tt->A_2-quot-ses}
	0\to A_2^\tt\to A_2\to A_2/A_2^\tt\to 0.
	\end{equation}
	By Proposition~\ref{p:-Att-for-abeloids}, the first term fits into a short exact sequence with torsion cokernel
	\[ 0\to \wh A^{\phantom{tt}}_2\to A_2^\tt\to A_2^\tt/\wh A^{\phantom{tt}}_2\to 0.\]
	By Corollary~\ref{c:H^1(wt A,wh A')}, this shows that the term $H^1_v(\wt A_1,A_2^\tt)$
	is coprime-to-$p$-torsion, and thus any morphism $h:\wt A_1\to A_2/A_2^\tt$ lifts to a morphism $\wt A_1\to A_2$ after composing with some $[N]:A_2/A_2^\tt\to A_2/A_2^\tt$. Since we already know that any morphism $\wt A_1\to A_2$ is the composition of a homomorphism with a translation, this is easily seen to be equivalent to saying that $h$ lifts after precomposing with $[N]:\wt A_1\to \wt A_1$.
	
	 It therefore suffices to prove that the composition
	\[ \overbar{\varphi}:\wtt A_1\to A_2\to A_2/A_2^\tt\]
	factors through $\wt A_1$ via some projection $\wtt A_1=\varprojlim_{[N]}\wt A_1\to \wt A_1$.
	
	We now again use an interpretation of this in terms of line bundles: By \cite[Theorem~4.2.1]{heuer-diamantine-Picard}, the sequence \eqref{eq:A_2^tt->A_2-quot-ses} injects into a left-exact sequence of diamantine Picard functors
	\[0\to A^\tt_2\to \uP_{A_2^\vee}\to  \uP_{\wtt A_2^\vee}.\]
	It follows that $\varphi$ corresponds to a line bundle $L$ on $\wtt A_1\times A_2^\vee$, and its image $\overbar{\varphi}$ in $\uP_{\wtt A_2^\vee}(\wtt A_1)$ corresponds to the pullback of $L$ to $\wtt A_1\times \wtt A_2^\vee$.
	
	But since $H^i_v(\wtt A_1\times \wtt A_2^\vee,\wh \G_m)= 1$ for $i\geq 1$ exactly as in Corollary~\ref{c:H^1(wt A,wh A')}, we see that
	\[ \Pic(\wtt A_1\times \wtt A_2^\vee)=H^1(\wtt A_1\times \wtt A_2^\vee,\bOx),\quad \Pic(\wt A_1\times \wtt A_2^\vee)=H^1(\wt A_1\times \wtt A_2^\vee,\bOx)\]
	where $\bOx$ is the sheaf from Definition~\ref{d:bOx}
	(see \cite[Proposition~4.6]{heuer-diamantine-Picard} for more details).
	By \cite[Corollary~3.26]{heuer-diamantine-Picard} and a \cH-argument, $\bOx$-cohomology can be approximated in limits, and we deduce:
	\[ \textstyle\Pic(\wtt A_1\times \wtt A_2^\vee)=\varinjlim_{[N]} \Pic(\wt A_1\times \wtt A_2^\vee).\]
	In terms of Picard functors, this means that the morphism $\overbar{\varphi}$ fits into a commutative square
	\[\begin{tikzcd}
		\wtt A_1 \arrow[r] \arrow[d,"\overbar{\varphi}"'] & \wt A_1 \arrow[d] \arrow[ld, dotted] \\
		A_2/A_2^\tt \arrow[r, hook] & \uP_{\wtt A_2},
	\end{tikzcd}\]
	and we thus obtain the dotted arrow using that $\wtt A_1\to \wt A_1$ surjects by the sequence \eqref{eq:ses-of-wttA}.
\end{proof}

\begin{proof}[Proof of Theorem~\ref{t:morphis-is-comp-of-hom-with-trans-adelic-case}]
	We need to show that every morphism $\wtt A_1\to \wtt A_2$ is the composition of a translation and a homomorphism. This is equivalent to the same property for $\wtt A_1\to A_2$. But this factors through $\wt A_1$ by Lemma~\ref{l:fact-of-morph-from-wttA}, so the statement follows from Theorem~\ref{t:main-theorem}.
\end{proof}

 The following pleasant special case was mentioned as Theorem~\ref{t:intro-adelic-covers-over-C_p} in the introduction:
\begin{Corollary}\label{c:adelic-covers-over-C_p}
	Let $K=\C_p$ or $\C_p^\flat$. Then the following are equivalent:
	\begin{enumerate}
	\item There is an isomorphism of perfectoid groups $\wtt A\isomarrow \wtt A'$ over $K$.
	\item There is an isomorphism of perfectoid spaces $\wtt A\isomarrow \wtt A'$ over $K$.
	\item There is an isogeny $\overbar B\to \overbar B'$ and $\rk(T)=\rk(T')$.
	\end{enumerate}
\end{Corollary}
\begin{Remark}
	This corollary only holds over $\C_p$ and  $\C_p^\flat$, and not over any non-trivial non-archimedean extension thereof: Indeed, we need both the value group to be $\Q$ and the residue field $k$ to have torsion unit group. If $K$ is a field in which either does not hold, there are already Tate curves whose adelic covers are non-isomorphic over $K$.
\end{Remark}

\begin{proof}[Proof of Corollary~\ref{c:adelic-covers-over-C_p}]
	If $K=\C_p$ or $\C_p^\flat$, then $k=\overline{\F}_p$, and thus $\overbar B(k)$ is a torsion group. It follows that  condition 3.(b) in Theorem~\ref{t:main-theorem-adelic-case} is vacuous.
	
	We are left to see that condition 3.(c) is equivalent to $\rk(T)=\rk(T')$. But by \cite[Lemma~B.7]{heuer-diamantine-Picard}, we have 
	\[(\wh E'(K)\cdot M')\otimes \Q=E'(K)\otimes \Q,\]
	so the condition is  satisfied up to translation by $M'\otimes \Q$. We can thus find an isomorphism $M\otimes \Q\to M'\otimes \Q$ such that the desired condition holds if and only if $\rk T=\rk T'$.
\end{proof}

\section{Open questions: Universal covers of curves}\label{s:curves}
Let $C$ be a connected smooth projective curve over $K$ of genus $g\geq 1$. Fix a base-point $x\in C(K)$. It was first observed by Hansen \cite{Hansen-blog}  that if $\Char K=0$,  there is a perfectoid universal pro-\'etale cover
\[ \wtt C\sim \varprojlim_{C'\to C}C'\]
where the index category consists of  all connected finite \'etale covers $(C',x')\to (C,x)$ equipped with a lift $x'$ of $x$. See \cite[Corollary~5.7]{perfectoid-covers-Arizona} for details. 

In the case of $\Char K=p$, one obtains an analogous construction by setting
\[ \wtt C\sim \varprojlim_{C'\to C}C'^{\perf}\]
with the same index category as above (see Definition~\ref{d:perf} for the perfection $-^\perf$).
Here we note that by \cite[Proposition 5.4.54]{GabberRamero} we have $C^\perf_{\fet}=C_{\fet}$, i.e. one could equivalently define $\wtt C$ by forming the limit over connected finite \'etale covers of $C^\perf$ with a choice of lift of $x'$.

In either case, $\wtt C$ is the (adelic) ``universal cover'' of $C$ in analogy to the cover $\wtt A\to A$ for abeloid varieties. Like in the abeloid case, we can characterise $\wtt C$ by saying that it has a universal property for pro-finite-\'etale morphisms over $C$, respectively $C^\perf$. Or one can characterise it by a topological universal property, analogously to Proposition~\ref{p:H^1_v(A,Zp)=0}:
\begin{Proposition}[{\cite[Proposition~2.1, Corollary~2.3]{heuer-Picard_perfectoid_covers}}]
	We have
	\[H^i_v(\wtt C,\underline{\whZ})=\begin{cases}\whZ&\quad \text{ for }i=0\\0& \quad \text{ for }i>0,\end{cases}\quad H^i_{v}(\wtt C,\O^+)\aeq \begin{cases}\O_K&\quad \text{ for }i=0\\0& \quad \text{ for }i>0.\end{cases}\]
\end{Proposition}

In analogy to \eqref{eq:A=wt A/T_pA}, the cover $\wtt C$ gives a uniformisation of $C$ as a diamond
\[ C^\diamondsuit = \wtt C/\pi_1(C,x)\]
where $\pi_1(C,x)$ is the \'etale fundamental group.

In the case $g=1$ of elliptic curves, this is precisely the (adelic) abeloid pro-\'etale uniformisation of $C$. In this case, Theorem~\ref{t:main-theorem} shows that the isomorphism class of $\wtt C$ is locally constant in the moduli space $\mathcal M_1(K)$ of elliptic curves. This motivates the following question that was independently also asked by Hansen (unpublished), and Litt:
\begin{Question}\label{q:isom-classes-of-covers-of-curves}
	Is the isomorphism class of $\wtt C$ locally constant in the moduli space $\mathcal M_g(K)$ of connected smooth projective curves of genus $g\geq 2$, with its natural topology on $K$-points? More precisely, if $K=\C_p$, then does  $\wtt C$ only depend on the semi-stable reduction of $C$?
\end{Question}

We note that this question has a very interesting connection to the Ehrenpreis conjecture in complex geometry, now a theorem by Kahn and Markovic \cite{ehrenpreis}, which roughly states that for any two compact Riemann surfaces $C_1$, $C_2$ of genus $g\geq 2$, there are finite degree covers $C_1'\to C_1$ and $C_2'\to C_2$ such that $C_1'$ and $C_2'$ are ``close'' in the complex moduli space. One could ask whether the following $p$-adic version of this statement holds:

\begin{Question}[Daniel Litt, \cite{Litt-MO-question}]\label{q:p-adic-Ehrenpreis}
	Let $n\in \N$. Given two connected smooth projective curves $C_1,C_2$ over $K$ of genus $g\geq 2$, are there finite \'etale covers $C_1'\to C_1$ and $C_2'\to C_2$ whose stable integral models become isomorphic mod $\varpi^n$? (Litt asks this for $K$ replaced by $\Z_p^{\mathrm{ur}}$).
\end{Question}
Should both questions have an affirmative answer, then $\wtt C$ would be isomorphic among all connected smooth projective curves of genus $\geq 2$. This would be an extremely close analogue of the Uniformisation Theorem in complex geometry saying that all compact Riemann surfaces of genus $\geq 2$ have isomorphic universal cover given by the upper half plane.
\appendix
\section{Appendix}
\subsection{Torsors, extensions, and the Breen--Deligne resolution}\label{s:BD}
We start by recalling a few basic notions on extensions in an abelian category $\mathcal A$.

\begin{Definition}
	Let $A,C\in \mathcal A$. By an extension of $C$ by $A$ we mean a short exact sequence $0\to A\to B\to C\to 0$ in $\mathcal A$. Let us write $\mathrm{EXT}(C,A)$ for the set of isomorphism classes of extensions. We write $[A\to B\to C]$, or just $[B]$, for the isomorphism class of the extension. 
\end{Definition}
\begin{Lemma}\label{l:interpretation-of-Ext1-as-extensions}
	Assume that $\mathcal A$ has enough injectives, then there is a natural identification
	\[ \mathrm{EXT}(C,A)=\Ext^1_{\mathcal C}(C,A).\]
	If $0\to A\to B \to C\to 0$ is any short exact sequence and $D\in \mathcal A$, then the boundary map
	\[ \Hom(A,D)\to\Ext^1_{\mathcal C}(C,D)\]
	is given by sending $\varphi:A\to D$ to the pushout of $[A\to B\to C]$ along $\varphi$.
\end{Lemma}

\begin{Remark}\label{r:not-enough-projectives-np-though}
	The category of abelian sheaves on a site always has enough injectives, so Lemma~\ref{l:interpretation-of-Ext1-as-extensions} applies in this situation. However, it rarely has enough projectives, so that one cannot define $\Ext^i(-,G)$ as a derived functor. Nevertheless, one can still prove ``by hand'' that any short exact sequence $0\to F_0\to F_1\to F_2\to 0$ gives rise to a long exact sequence of $\Ext^i(-,G)$ for any $G$: One takes an injective resolution $G\to I^\bullet$ and applies $\Hom(-,I^\bullet)$.
\end{Remark}

\begin{Definition}
	Let $\mathcal C$ be a site, let $X$ be a sheaf on $\mathcal C$ and let $G$ be  a sheaf of groups on $\mathcal C$. Then a $G$-torsor on $X$ is a sheaf of sets $f: F\to X$ together with an action $m:G\times  F\to F$ fixing $f$ for which there is a surjection $Y\to X$ such that the pullback $f|_Y$ is isomorphic to $G\times Y\to Y$ with its $G$-action via the first factor. We recall \cite[03AG]{StacksProject}:
\end{Definition}
\begin{Lemma}
	The group $H^1_{\mathcal C}(X,G)$ classifies $G$-torsors on $X$ up to isomorphism.
\end{Lemma}

Let now $\mathcal C$ be a site and let $A,C$ be abelian sheaves on $\mathcal C$. Then for every extension
\[0\to A\to B\to C\to 0, \]
we can regard $B$ as an $A$-torsor on $C$ in a natural way. This defines a map
\[ \Ext^1(C,A)\to H^1(C,A),\]
which can be interpreted in cohomological terms as follows:
\begin{Definition}
	For any sheaf of sets $X$ on $\mathcal C$, we write $\Z[X]$ for the sheaf of sets given by sheafifying the presheaf $\Z[X](U)=\Z[X(U)]$. This defines a functor $\Z[-]$ from sheaves of sets to sheaves of abelian groups. For any sheaf of sets $X$, we have a natural equivalence $\Hom(\Z[X],-)=\Gamma(X,-)$ which upon derivation gives
	\[\RHom(\Z[X],-)= \RGamma(X,-).\]
\end{Definition}
Let now $X$ be a sheaf of abelian groups, then composition with the pullback along the counit map $\Z[X]\to X$ yields a natural transformation
\begin{equation}\label{eq:ext-to-torsors}
	\Phi:\Ext^1(X,-)\to \Ext^1(\Z[X],-)=H^1(X,-),
\end{equation}
which is the homological interpretation of the above map of sets of isomorphism classes.

The map $\Phi$ in \eqref{eq:ext-to-torsors} is in general neither injective nor surjective: An extension $[B]$ might be split as a torsor, but not in a way respecting the group structures. On the other hand, there might be non-trivial torsors that do not admit a group structure.

We now recall an explicit functorial way in which one can describe the kernel of $\Phi$ using the Breen--Deligne resolution (\cite[\S2.1]{BBM}, see \cite[Appendix to IV]{Scholze-Condensed} for a proof).
\begin{Theorem}[Breen--Deligne]\label{t:BD}
	For abelian sheaves $G,F$ on $\mathcal C$, there is a functorial resolution
	\[ C(G)_{\bullet}:=(\dots \to \Z[X_1]\to \Z[X_0])\cong G \]
	where each $X_i$ is a finite product of copies of $G$. It gives rise to a functorial spectral sequence
	\[E_1^{ij}(G,F):=\Ext^j(C(G)_i,F)=H^j(X_i,F)\Rightarrow \Ext^{i+j}(G,F)\]
	for any abelian sheaf $F$. The associated five-term exact sequence is given by
	\[0\to E_2^{10}(G,F)\to \Ext^1(G,F)\xrightarrow{\Phi} E_2^{01}=\ker(H^1(G,F)\xrightarrow{d_1} H^1(G\times G,F)) \to E_2^{20}(G,F), \]
	where $d_1$ is the map induced by $\Z[G]^2\to \Z[G], [x,y]\mapsto [x+y]-[x]-[y].$
\end{Theorem}
We note that the $0$-th row $E_1^{\bullet 0}(G,F)$ of the spectral sequence reads
\begin{equation}\label{eq:BD-first-row}
	E_1^{00}=H^0(G,F)\xrightarrow{d_1} E_1^{10}=H^0(G^2,F)\xrightarrow{d_2} E_1^{20}=H^0(G^3,F)\oplus H^0(G^2,F)\xrightarrow{d_3} \dots
\end{equation}
where $d_1$ is like above, and there are explicit formulas for $d_2$ and $d_3$, see \cite[\S  2.1.5]{BBM}.

\begin{Lemma}\label{l:E_1-insensitive-to-topology}
	Let $\mathcal C$ be any site and let $\mathcal C'$ be a site with the same underlying category but a finer topology. Assume that $\mathrm{Sh}_{\mathcal C'}\to \mathrm{Sh}_{\mathcal C}$ is fully faithful.  Then for any abelian sheaves $F,G$ on $\mathcal C'$, the complex $E_1^{\bullet 0}(G,F)$ is the same whether it is computed with respect to $\mathcal C$ or $\mathcal C'$.
\end{Lemma}
\begin{proof}
	As we see from Theorem~\ref{t:BD}, we only need $H^0$ to compute $E_1^{\bullet 0}$. The fully faithfulness assumption guarantees that this is the same for $\mathcal C$ and $\mathcal C'$.
\end{proof}

\subsection{$p$-divisible sheaves}
\begin{Definition}\label{d:p-div-sheaf}
	\begin{enumerate}
		\item An abelian sheaf $F$ on a site $\mathcal C$ is called $p$-divisible if $[p]:F\to F$
		is surjective. It is called uniquely $p$-divisible if $[p]$ is an isomorphism.
		\item A sheaf $F$ on $\mathcal C$ is called derived $p$-complete if $\Rlim_n H/p^n=H$.
	\end{enumerate}
\end{Definition}

\begin{Lemma}\label{l:ext-from-p-divisible-to-complete}
	Let $G,H$ be abelian sheaves on a site $\mathcal C$. Assume that $G$ is uniquely p-divisible and $H$ is derived $p$-complete. Then
	$\Ext^i(G,H)=0 \text{ for all }i\geq 0$.
\end{Lemma}
\begin{proof}
	Since each $G$ is uniquely $p$-divisible, so is each $\Ext^i(G,-)$, and thus
	\[\RHom(G,H)=\RHom(G,\Rlim H/p^n)=\Rlim \RHom(G,H/p^n)=0.\qedhere\] 
\end{proof}

\begin{Proposition}\label{p:higher-cohomolgy-of-profinite=0}
	
	\begin{enumerate}
	\item\label{enum:H^i(profinite,profinite)} Let $H$ be a profinite topological space and let $G$ be a locally profinite abelian topological group, then for any $i>0$, we have 
	$H^i_v(\underline H,\underline{G})=0$.
	\item\label{enum:Ext(constant,profinite)} If $G$ is discrete and $H=\varprojlim H_j$ is a profinite topological group, then for $i\geq 0$: \[\Ext^i_v(\underline{H},\underline{G})=\textstyle\varinjlim_{j} \Ext^i_{\Z}(H_j,G).\]
	\end{enumerate}
\end{Proposition}
\begin{proof}
	For part 1, since any locally profinite group is an extension of a discrete group by a profinite group, we can  reduce to $G$ discrete or profinite. We first assume that $G=\varprojlim G_k$ is profinite.
	Then $\underline{G}=\Rlim \underline{G}_k$ by Lemma~\ref{l:qproet-replete}. 
	The Grothendieck spectral sequence for $\lim \circ \Gamma(\underline{H},-)$ thus gives for $i>0$ a short exact sequence
	\[0\to R^1\varprojlim  H^{i-1}_{v}(\underline{H},\underline{G}_k)\to H^{i}_{v}(\underline{H},\underline{G})\to \varprojlim  H^i_{v}(\underline{H},\underline{G}_k)\to 0. \]
	For $i=1$, the first term vanishes since $H^0(\underline{H},\underline{G}_k)=\Map_{\cts}(H,G_k)$ is a Mittag-Leffler system by Lemma~\ref{l:underline-is-left-adjoint}. For general $i$, this reduces us to the case that $G$ is discrete.
	
	In this case, $\underline{G}$ is the pullback of the locally constant sheaf associated to $G$ on $\underline{H}_{\et}$.
	By \cite[Propositions 14.7 and 14.8]{etale-cohomology-of-diamonds}, we thus have $H^i_v(\underline{H},\underline{G})=H^i_{\et}(\underline{H},\underline{G})$.
	But by \cite[Proposition 7.16]{etale-cohomology-of-diamonds}, the perfectoid space $\underline{H}$ is strictly totally disconnected, i.e.\ every \'etale cover of $\underline{H}$ splits. Thus $H^i_{\et}(\underline{H},F)=0$ holds for any abelian sheaf $F$ on $\underline{H}_{\et}$.
	
	We deduce part 1 from this using Theorem~\ref{t:BD}: By part 1, the Breen--Deligne spectral sequence degenerates and gives an isomorphism for any $i\geq 0$:
	\[ \Ext^i_v(\underline{H},\underline{G})=E_2^{i0}(\underline{H},\underline{G}).\] 
	The latter is by definition $H^i(E_1^{\bullet0})$ for the complex $E_1^{\bullet0}$ which consists of terms of the form
	\[E_1^{i0}=\oplus_{k}H^0(\underline{H}^{n_k},\underline{G})=\oplus_{k}\Map_{\cts}(H^{n_k},G)=\textstyle\varinjlim_j\oplus_{k}\Map_{\cts}(H_j^{n_k},G),\]
	where the second equality uses Lemma~\ref{l:underline-is-left-adjoint}. This reduces us to the case that $H$ is discrete.
	
	If $H$ is discrete, we see from this equation that
	$E_1^{\bullet 0}$ agrees with the complex that computes $\Ext^{i}_{\mathcal C}(H,G)$ for the site $\mathcal C$ that consists of just a point, and whose category of abelian sheaves is equivalent to abelian groups. It follows that  $H^i(E_1^{\bullet 0})=\Ext^i_{\Z}(H,G)$, as desired.
\end{proof}

\begin{Lemma}\label{l:E_2-of-discrete-torsor-over-connected}
	Let $G$ be a connected adic group and let $H$ be a locally profinite group, both abelian. Then $E_2^{\bullet 0}=0$. In particular, the map $\Phi:\Ext^1_v(G,\underline{H})\hookrightarrow H^1_v(G,\underline{H})$ is injective.
\end{Lemma}
\begin{proof}
	We learnt the following argument from \cite[\S4.5]{A-LB}:
	We need to see that the complex $E_1^{\bullet 0}$ in  \eqref{eq:BD-first-row} is exact. Since $G$ is connected and $H$ is locally profinite, any morphism
	$G^n\to \underline{H}$
	is necessarily constant, so that $H^0(G^n,\underline{H})=H$. By functoriality of~\eqref{eq:BD-first-row}, this shows that the sequence is identical for \textit{any} connected adic group. In particular, this holds for the trivial group $G:=0=\Spa(K,\O_K)$, in which case clearly $\Ext(0,\underline{H})=0$, thus $E_2^{\bullet 0}=0$.
	
	The injection now follows from the 5-term exact sequence in Theorem~\ref{t:BD}.
\end{proof}
\bibliographystyle{abbrv}
\bibliography{../../universal.bib}
\end{document}